\documentclass[letterpaper,11pt]{article}

\usepackage[small]{titlesec}
\usepackage{theorem,amsmath,amssymb,amscd, verbatim}
\usepackage[all]{xy}
\usepackage{booktabs}
\usepackage[hyperfootnotes=false, colorlinks, linkcolor={blue}, citecolor={magenta}, filecolor={blue}, urlcolor={blue}]{hyperref}
\theoremstyle{change}
\allowdisplaybreaks
\newcommand{\leftexp}[2]{{\vphantom{#2}}^{#1}%
      \kern-\scriptspace%
      {#2}}
\renewcommand{\Im}{\mathrm{Im}}
\renewcommand{\Re}{\mathrm{Re}}
\renewcommand{\H}{\mathbb H}

\newcommand{\A}{{\mathbb A}}

\newcommand{\D}{{\mathcal D}}

\newcommand{\Q}{{\mathbb Q}}
\newcommand{\Z}{{\mathbb Z}}

\newcommand{\R}{{\mathbb R}}
\newcommand{\C}{{\mathbb C}}
\newcommand{\bs}{\backslash}

\newcommand{\p}{\mathfrak p}

\newcommand{\OF}{{\mathfrak o}}
\newcommand{\GL}{{\rm GL}}

\newcommand{\GSpin}{{\rm GSpin}}
\newcommand{\SO}{{\rm SO}}

\newcommand{\GSp}{{\rm GSp}}
\newcommand{\Sp}{{\rm Sp}}

\newcommand{\PGSp}{{\rm PGSp}}
\newcommand{\Aut}{{\rm Aut}}

\newcommand{\sgn}{{\rm sgn}}

\newcommand{\vol}{{\mathrm {vol}}}

\newcommand{\sym}{{\rm sym}}
\newcommand{\vl}{{\rm vol}}

\newcommand{\mat}[4]{{\setlength{\arraycolsep}{0.5mm}\left[
\begin{smallmatrix}#1&#2\\#3&#4\end{smallmatrix}\right]}}
\newcommand{\qed}{\hspace*{\fill}\rule{1ex}{1ex}}
\newcommand{\forget}[1]{}

\def\qdots{\mathinner{\mkern1mu\raise0pt\vbox{\kern7pt\hbox{.}}\mkern2mu
\raise3.4pt\hbox{.}\mkern2mu\raise7pt\hbox{.}\mkern1mu}}

\newenvironment{proof}{\vspace{1ex}\noindent{\it Proof.}\hspace{0.1em}}
	{\hfill\qed\vspace{2ex}}

\newtheorem{lemma}{Lemma.}[section]
\newtheorem{theorem}[lemma]{Theorem.}
\newtheorem{corollary}[lemma]{Corollary.}
\newtheorem{proposition}[lemma]{Proposition.}
\newtheorem{definition}[lemma]{Definition.}
\newtheorem{remark}[lemma]{Remark.}

\begin{document}

\bibliographystyle{plain}
\thispagestyle{empty}
\begin{center}
 {\bf\Large Integral representation and critical $L$-values for holomorphic forms on $\GSp_{2n} \times \GL_1$}

 \vspace{3ex}
 Ameya Pitale, Abhishek Saha\footnote{A.S. acknowledges the support of the EPSRC grant EP/L025515/1.}, Ralf Schmidt
\end{center}

\begin{abstract}We prove an  explicit integral representation -- involving the pullback of a suitable Siegel Eisenstein series -- for the twisted standard $L$-function associated to a holomorphic vector-valued Siegel cusp form of degree $n$ and arbitrary level. In contrast to all previously proved pullback formulas in this situation, our  formula involves only scalar-valued functions despite being applicable to $L$-functions of vector-valued Siegel cusp forms.
The key new ingredient in our method is a novel choice of local vectors at the archimedean place which allows us to exactly compute the archimedean local integral.

By specializing our integral representation to the case $n=2$, we are able to prove a
reciprocity law -- predicted by Deligne's conjecture -- for the critical special values of the twisted standard $L$-function for vector-valued Siegel cusp forms of degree 2 and arbitrary level. This arithmetic application generalizes previously proved critical-value results for the full level case. The  proof of this application uses our recent structure theorem \cite{PSS14} for the space of nearly holomorphic
 Siegel modular forms of degree 2 and arbitrary level.

\end{abstract}
\tableofcontents

\section{Introduction}
\subsection{Critical \texorpdfstring{$L$}{}-values}The critical values of $L$-functions attached to cohomological cuspidal automorphic representations are objects of deep arithmetic significance. In particular, a famous conjecture of Deligne \cite{deligneconj} predicts that these values are algebraic numbers up to multiplication by suitable periods\footnote{However, it is often a non-trivial problem to relate the automorphic periods to Deligne's motivic periods.}. One can go even further and ask for the prime ideal factorizations of these algebraic numbers (once the period is suitably normalized). This goes beyond Deligne's conjecture and is related to the Bloch-Kato conjectures.

The first results on the algebraicity of critical $L$-values of automorphic representations were obtained by Shimura \cite{shi59, shi75, shi76}, who considered automorphic representations of $\GL_2$ (and twists and Rankin-Selberg products of these). For higher rank groups, initial steps were taken by Harris \cite{harsieg} and Sturm \cite{sturm} in 1981, who considered automorphic representations of $\GSp_{2n} \times \GL_{1}$ whose finite part is unramified and whose infinity type is a holomorphic discrete series with scalar minimal $K$-type. Since then, there has been considerable work in this area and algebraicity results for $L$-values of automorphic representations on various algebraic groups have been proved; the general problem, however, remains very far from being resolved.

In this paper we revisit the case of the standard $L$-function on $\GSp_{2n} \times \GL_{1}$. As mentioned earlier, this is the first case outside $\GL_2$ that was successfully tackled, with the (independent) results of Harris and Sturm back in 1981. However, the automorphic representations considered therein were very special and corresponded, from the classical point of view, to scalar-valued Siegel cusp forms of full level. Subsequent works on the critical $L$-values of Siegel cusp forms by B\"ocherer \cite{boch1985}, Mizumoto \cite{miz}, Shimura \cite{shibook2}, Kozima \cite{kozima} and others strengthened and extended these results in several directions. Nonetheless, a proof of algebraicity of critical $L$-values for  holomorphic forms on $\GSp_{2n} \times \GL_{1}$ in full generality has not yet been achieved, \emph{even for $n=2$}.

We prove the following result for $n=2$, which applies to representations whose finite part is arbitrary and whose archimedean part can be (almost) any holomorphic
discrete series.\footnote{We omit certain holomorphic discrete series representations of ``low weight", namely those with $k_2<6$, as well as the ones that do not contribute to cohomology, namely those with $k_1 \not\equiv k_2 \pmod{2}$.}
\begin{theorem}\label{t:mainintro}
 Let $k_1 \ge k_2 \ge 6$, $k_1 \equiv k_2 \pmod{2}$ be integers. For each cuspidal automorphic representation $\pi$ on $\GSp_4(\A_\Q)$ with $\pi_\infty$ isomorphic to the holomorphic discrete series representation with highest weight $(k_1, k_2)$, there  exists a real number $C(\pi)$ with the following properties.
 \begin{enumerate}
  \item $C(\pi) = C(\pi')$ whenever $\pi$ and $\pi'$ are nearly equivalent.
  \item $C(\pi) = C(\pi  \otimes (\psi\circ\mu))$ for any Dirichlet character $\psi$. Here $\mu: \GSp_4 \rightarrow \GL_1$ denotes the standard multiplier map.
  \item For  any Dirichlet character $\chi$ satisfying $\chi(-1) = (-1)^{k_2}$, any finite subset $S$ of the places of $\Q$ that includes the place at infinity, any integer $r$ satisfying $4 \le r \le k_2-2$, $r \equiv k_2 \pmod{2}$,  and any $\sigma \in \Aut(\C)$, we have
  \begin{equation}\label{e:mainarith}
   \sigma \left(\frac{L^S(r,\pi \boxtimes \chi, \varrho_{5})}{(2\pi i)^{3r} G(\chi) G(\chi^2)^2 C(\pi)}\right) = \frac{L^S(r,{}^\sigma\!\pi \boxtimes{}^\sigma\!\chi, \varrho_{5})}{(2\pi i)^{3r} G({}^\sigma\!\chi) G({}^\sigma\!\chi^2)^2 C({}^\sigma\!\pi)},
  \end{equation}
  and therefore in particular
  $$
   \frac{L^S(r,\pi \boxtimes \chi, \varrho_{5})}{(2\pi i)^{3r} G(\chi) G(\chi^2)^2 C(\pi)} \in \Q(\pi, \chi).
  $$
  Above,  $G(\chi)$  denotes the Gauss sum, $L^S(s,\pi \boxtimes \chi, \varrho_{5})$ denotes the degree $5$ $L$-function (after omitting the local factors in $S$) associated to the representation $\pi \boxtimes \chi$ of $\GSp_4 \times \GL_{1}$, and $\Q(\pi, \chi)$ denotes the (CM) field generated by $\chi$ and the field of rationality for $\pi$.
\end{enumerate}
\end{theorem}
Before going further, we make a few remarks pertaining to the statement above.

\begin{remark}Note that we can take $S = \{\infty\}$. In particular this gives us an algebraicity result for the full finite part of the global $L$-function.
\end{remark}

\begin{remark}Our proofs show that $C(\pi) = \pi^{2k_1} \langle F, F\rangle$ where $F$ equals a certain \textbf{nearly holomorphic} modular form of weight $k_1$. Alternatively one can take $C(\pi) = \pi^{k_1+k_2} \langle F_0, F_0\rangle$ where $F_0$ equals a certain holomorphic \textbf{vector-valued} modular form of weight $\det^{k_2}\sym^{k_1-k_2}$.
\end{remark}

\begin{remark}Observe that the $L$-function $L(s,\pi \boxtimes \chi, \varrho_{5})$ considered above is \emph{not} equal to the $L$-function $L(s,\pi \otimes (\chi\circ \mu), \varrho_{5})$. In fact, $L(s,\pi \boxtimes \chi, \varrho_{5})$ cannot be obtained as the $L$-function of an automorphic representation on $\GSp_4(\A_\Q)$, unless $\chi$ is trivial.
\end{remark}

Classically, Theorem \ref{t:mainintro} applies to vector-valued holomorphic Siegel cusp forms of weight $\det^{k_2} \sym^{k_1 - k_2}$ with respect to an arbitrary congruence subgroup of $\Sp_4(\Q)$. The only previously known result for critical $L$-values of holomorphic Siegel cusp forms in the vector-valued case ($k_1 > k_2$) is due to Kozima \cite{kozima}. Kozima's result only applies to full-level Siegel cusp forms (and also only deals with the case $\chi=1$). In contrast, our theorem, which relies on an adelic machinery to separate out the difficulties place by place, is more general in that it applies to arbitrary congruence subgroups.

In fact, our theorem appears to be new even in the scalar-valued case $k_1=k_2$, as all previously proved special value results for this $L$-function (see, e.g., \cite{shibook2}) assumed that the Siegel cusp form is with respect to a $\Gamma_0(N)$-type congruence subgroup.

The main achievement of this paper, however, is \emph{not} Theorem \ref{t:mainintro} above but an explicit integral representation (Theorem \ref{t:intrepintro} below) for the standard $L$-function $L(s,\pi \boxtimes \chi, \varrho_{2n+1})$ on $\GSp_{2n} \times \GL_1$. This integral representation applies to any $\pi \boxtimes\chi$ such that $\pi_\infty$ is a holomorphic discrete series representation with highest weight $(k_1, k_2, \ldots, k_n)$ with all $k_i$ of the same parity, and $\chi_\infty = \sgn^{k_1}$ (there is no restriction on the finite part of $\pi$). While Theorem \ref{t:intrepintro} is formally similar to the well-known pullback formula (or doubling method) mechanism due to Garrett, Shimura, Piatetski-Shapiro--Rallis and others, what distinguishes it from previous works is that (despite the generality of our setup) all our constants are completely explicit. Previously, such an explicit pullback formula in this case  had been proved only for very special archimedean types. The key innovation that makes Theorem \ref{t:intrepintro} possible is a certain novel choice of  local vectors which allows us to exactly compute all local zeta integrals appearing in the pullback formula, including at the ramified and archimedean places. Note that Theorem \ref{t:intrepintro} applies to any $n$; we restrict to $n=2$ only in the final section of this paper where we prove Theorem \ref{t:mainintro}.

In the rest of this introduction, we explain the above ideas in more detail.

\subsection{Integral representations for \texorpdfstring{$\GSp_{2n} \times \GL_1$}{} and the pullback formula}\label{s:intropullback}

The first integral representation for the standard (degree $2n+1$) $L$-function for automorphic representations of $\GSp_{2n} \times \GL_1$ was discovered by Andrianov and Kalinin \cite{andkal} in 1978. The integral representation of Andrianov--Kalinin applied to holomorphic scalar-valued Siegel cusp forms of even degree $n$ with respect to $\Gamma_0(N)$ type congruence subgroups, and involved a certain theta series. They used their integral representation to prove results about the meromorphic continuation of the $L$-function; they also derived the functional equation when 8 divides $n$. The results of Harris \cite{harsieg} and Sturm \cite{sturm} mentioned earlier also relied on this integral representation.

A remarkable new integral representation, commonly referred to as the pullback formula, was discovered in the early 1980s by Garrett \cite{gar83} and involved pullbacks of Eisenstein series. Roughly speaking, the pullback formula in the simplest form says that \begin{equation}\label{garrpullback}
 \int\limits_{\Sp_{2n}(\Z) \bs \H_n} F(-\overline{Z_1}) E_k\left(\mat{Z_1}{}{}{Z_2}, s    \right) dZ_1\approx L(2s+k-n,\pi,  \varrho_{2n+1})F(Z_2),
\end{equation}
where $F$ is a Siegel cusp form of degree $n$ and full level that is an eigenform for all the Hecke operators, $E_k\left(Z, s    \right)$ is an Eisenstein series of degree $2n$ and full level (that is, a holomorphic Siegel modular form of weight $k$ when $s=0$), $L(s,\pi,  \varrho_{2n+1})$ denotes the degree $2n+1$ $L$-function for $\pi$, and the symbol $\approx$ indicates that we are ignoring some unimportant factors.
The pullback formula was applied by B\"ocherer \cite{boch1985} to prove various results about the functional equation and algebraicity of critical $L$-values that went well beyond what had been possible by the Andrianov--Kalinin formula. Subsequently, Shimura generalized the pullback formula to a wide variety of contexts (including other groups). We refer to Shimura's books \cite{shibook1, shibook2} for further details.

In the last two decades, the pullback formula has been used to prove a host of results related to the arithmetic and analytic properties of $L$-functions associated to Siegel cusp forms. However, most of these results involve significant restrictions. To get the most general results possible, it is necessary to extend \eqref{garrpullback} to a) incorporate \emph{characters}, b) include Siegel cusp forms with respect to arbitrary \emph{congruence subgroups}, and c) cover the case of \emph{vector-valued} Siegel cusp forms. The first two of these objectives have to a large extent been achieved for scalar-valued Siegel cusp forms. However, the situation with vector-valued forms is quite different. Following the important work of B\"ocherer--Satoh--Yamazaki \cite{bochsatyama}, there have been a few results about vector-valued forms  by Takei \cite{takei92},  Takayanagi \cite{takayanagi93, takayanagi95}, Kozima \cite{kozima, kozima08}, and others. Unfortunately, all these works are valid only for full level Siegel cusp forms and involve strong restrictions on the archimedean type.\footnote{There is also a survey talk by B\"ocherer (Hakuba, 2012) on the pullback formula for more general vector-valued forms, but it still restricts to full level and does not incorporate characters.} To get the most general pullback formula for vector-valued Siegel cusp form, it seems necessary to use the adelic approach. All the results about vector-valued forms described above use certain differential operators that take scalar-valued functions to tensor products of vector-valued functions \cite{bochsatyama, ibu99} and it seems unclear if these can be incorporated into an adelic treatment.

On the other hand, Piatetski-Shapiro and Rallis \cite{psrallis83, psrallis87}  discovered a very general  identity (the doubling method) on classical groups. When the group is $\Sp_{2n}$, this is essentially a generalized, adelic version of the pullback formula described above.\footnote{However, at least in the original formulation, the doubling method did not incorporate characters.} However,  Piatetski-Shapiro and Rallis computed the relevant local integrals only at the unramified places (where they chose the vectors to be the unramified vector). Thus, to get an explicit integral representation from the basic identity of Piatetski-Shapiro and Rallis for more general automorphic representations, it is necessary to make specific choices of vectors at the ramified and archimedean primes such that one can exactly evaluate all the associated local integrals. So far, this has been carried out in very few situations.

In this paper, we begin by reproving the basic identity of Piatetski-Shapiro and Rallis in a somewhat different manner that is more convenient for our purposes. Let $\pi$ be a cuspidal automorphic representation of $\GSp_{2n}(\A)$ and
$\chi$ a Hecke character.
 Our central object of investigation is the global integral
\begin{equation}\label{maindefinition1}
 Z(s; f, \phi)(g) = \int\limits_{\Sp_{2n}(F) \bs\,g\cdot\Sp_{2n}(\A)} E^{\chi}((
 h  ,g), s, f) \phi(h)\,dh.
\end{equation}
Here, $E^{\chi}( - ,s, f)$ is an Eisenstein series on $\GSp_{4n}(\A)$ associated to a choice of section $f \in I(\chi, s)$, the pair $(h,g)$ represents an element of $\GSp_{4n}(\A)$ corresponding to the diagonal embedding of  $\GSp_{2n}(\A) \times  \GSp_{2n}(\A)$, and $\phi$ is an automorphic form in the space of $\pi$. The integral \eqref{maindefinition1} represents a meromorphic function of $s$ on all of $\C$ due to the analytic properties of the Eisenstein series. As we will prove, away from any poles, the function $g\mapsto Z(s; f, \phi)(g)$ is again an \emph{automorphic form in the space of $\pi$.}

Next, assume that $\phi = \otimes_v \phi_v$ and $f=\otimes f_v$ are factorizable. We define, for each place $v$, \emph{local zeta integrals}
\begin{equation}\label{Rsflocaleq1}
 Z_v(s;f_v,\phi_v):=\int\limits_{\Sp_{2n}(F_v)} f_v(Q_n \cdot (h, 1), s) \pi_v(h)\phi_v\,dh,
\end{equation}
where $Q_n$ is a certain explicit matrix in $\GSp_{2n}(\A)$. It turns out (see Proposition \ref{interopprop}) that $Z_v(s;f_v,\phi_v)$ converges to an element in the space of $\pi_v$ for real part of $s$ large enough.

Our ``Basic Identity'' (Theorem \ref{basicidentity}) asserts that the automorphic forms $Z(s;f, \phi)$ corresponds to the pure tensor $\otimes_v Z_v(s;f_v, \phi_v)$.
Now assume that, for all $v$, the vectors $\phi_v$ and the sections $f_v$ can be chosen in such a way that $Z_v(s,f_v,\phi_v)=c_v(s)\phi_v$ for an explicitly computable function $c_v(s)$. Then we obtain a \emph{global integral representation}
\begin{equation}\label{globalintrepeq1}
 Z(s; f, \phi)(g)=\Big(\prod_vc_v(s)\Big)\phi.
\end{equation}
In this way the Euler product $\prod_vc_v(s)$, convergent for ${\rm Re}(s)$ large enough, inherits analytic properties from the left hand side of \eqref{globalintrepeq1}. If this Euler product, up to finitely many factors and up to ``known'' functions, represents an interesting $L$-function, one can thus derive various properties of said $L$-function.

Our main local task is therefore to choose the vectors $\phi_v$ and the sections $f_v$ such that $Z_v(s,f_v,\phi_v)=c_v(s)\phi_v$ for an explicitly computable function $c_v(s)$. For a ``good'' non-archime\-dean place $v$ we will make the obvious unramified choices. The \emph{unramified calculation}, Proposition \ref{unramifiedcalculationprop}, then shows that
\begin{equation}\label{unramcalceq}
 c_v(s)=\frac{L((2n+1)s+1/2,\pi_v \boxtimes \chi_v, \varrho_{2n+1})}{L((2n+1)(s+1/2), \chi_v)\prod\limits_{i=1}^nL((2n+1)(2s+1)-2i, \chi_v^2)}.
\end{equation}
If $v$ is non-archimedean and some of the data is \emph{ramified}, it is possible to choose $\phi_v$ and $f_v$ such that $c_v(s)$ is a non-zero constant function; see Proposition \ref{prop:badplaces}. The idea is to choose $\phi_v$ to be invariant under a small enough principal congruence subgroup, and make the support of $f_v$ small enough. This idea was inspired by \cite{morimoto}.

\subsection{The choice of archimedean vectors and our main formula}
We  now explain our choice of $\phi_v$ and $f_v$ at a real place $v$, which represents one of the main new ideas of this paper. We  only treat the case of $\pi_v$ being a holomorphic discrete series representation, since this is sufficient for our application to Siegel modular forms. Assume first that $\pi_v$ has scalar minimal $K$-type of weight $k$, i.e., $\pi_v$ occurs as the archimedean component attached to a classical Siegel cusp form of weight $k$. Then it is natural to take for $\phi_v$ a vector of weight $k$ (spanning the minimal $K$-type of $\pi_v$), and for $f_v$ a vector of weight $k$ in $I(\chi_v,s)$. Both $\phi_v$ and $f_v$ are unique up to scalars. The local integral in this case is calculated in Proposition \ref{scalarminKtypeprop} and turns out to be an exponential times a rational function. The calculation is made possible by the fact that we have a simple formula for the matrix coefficient of $\phi_v$; see \eqref{matrixcoeffeq}.

Now assume that $\pi_v$ is a holomorphic discrete series representation with more general minimal $K$-type $\rho_{(k_1,\ldots,k_n)}$, where $k_1\geq\ldots\geq k_n>n$; see Sect.~\ref{holdiscsec} for notation. In this case we will \emph{not} work with the minimal $K$-type, but with the scalar $K$-type $\rho_{(k,\ldots,k)}$, where $k:=k_1$. We will show in Lemma \ref{scalarKtypeslemma} that $\rho_{(k,\ldots,k)}$ occurs in $\pi_v$ with multiplicity one. Let $\phi_v$ be the essentially unique vector spanning this $K$-type, and let $f_v$ again be the vector of weight $k$ in $I(\chi_v,s)$. The function $c_v(s)$ in this case is again an exponential times a rational function; see Proposition \ref{archzetaprop}. We note that our calculation for general minimal $K$-type uses the result of the scalar case, so the scalar case cannot be subsumed in the more general case. One difficulty in the general case is that we do not know a formula for the matrix coefficient of $\phi_v$, or in fact \emph{any} matrix coefficient. Instead, we use a completely different method, realizing $\phi_v$ as a vector in an induced representation.

Our archimedean calculations require two different integral formulas, which are both expressions for the Haar measure on $\Sp_{2n}(\R)$. The first formula, used in the scalar minimal $K$-type case, is with respect to the $KAK$ decomposition; see \eqref{KAKintegrationeq}. The second formula, used in the general case, is with respect to the Iwasawa decomposition; see \eqref{Sp2niwasawainteq2}. We will always normalize the Haar measure on $\Sp_{2n}(\R)$ in the ``classical'' way, characterized by the property \eqref{Ghaarpropeq1}. It is necessary for us to determine the precise constants relating the $KAK$ measure and the Iwasawa measure to the classical normalization. This will be carried out in Appendix \ref{measureapp}.

Finally, combining the archimedean and the non-archimedean calculations, we obtain an explicit formula for the right side of \eqref{globalintrepeq1}, which is our aforementioned pullback formula for $L$-functions on $\GSp_{2n} \times \GL_1$. This is Theorem \ref{global-thm} in the main body of the paper; below we state a rough version of this theorem.
\begin{theorem}\label{t:intrepintro}Let $\pi\cong\otimes \pi_p$ be an irreducible cuspidal automorphic representation of \linebreak $\GSp_{2n}(\A_\Q)$ such that $\pi_\infty$ is a holomorphic discrete series representation of highest weight $(k_1,k_2, \ldots, k_n)$, where $k=k_1\ge\ldots\ge k_n>n$ and all $k_i$ have the same parity. Let $\chi=\otimes\chi_p$ be a character of $\Q^\times \bs \A^\times$ such that $\chi_\infty = \sgn^k$.  Then we have the identity
\begin{align*}\label{global-int-formula}
  Z(s, f, \phi)(g) =  &\frac{L^S((2n+1)s+1/2,\pi \boxtimes \chi, \varrho_{2n+1})}{L^S((2n+1)(s+1/2), \chi)\prod_{j=1}^nL^S((2n+1)(2s+1)-2j, \chi^2)} \nonumber \\ &\times  i^{nk}\,\pi^{n(n+1)/2}  c((2n+1)s- 1/2)\phi(g),
 \end{align*}
  where $c(z)$ is an explicit function that takes non-zero rational values for integers $z$ with $0 \le z \le k_n-n$ and
  where $L^{S}(\ldots)$ denotes the global $L$-functions without the local factors corresponding to a certain (explicit) finite set of places $S$.
\end{theorem}

The above formula can be reformulated in a classical language, which takes a similar form to \eqref{garrpullback} and involves functions $F(Z)$ and $E_{k,N}^\chi(Z, s)$ that correspond to $\phi$ and  $E^{\chi}( - ,s,f)$ respectively. In fact $F$ and $E_{k,N}^\chi( - , s)$ are (scalar-valued) smooth modular forms of weight $k$ (and degrees $n$ and $2n$ respectively) with respect to suitable congruence subgroups. We refer the reader to Theorem \ref{classical-integral-repn} for the full classical statement.  In  all previously proved classical pullback formulas \cite{bochsatyama, takayanagi93, takayanagi95} for $L(s, \pi \boxtimes \chi)$ with $\pi_\infty$ a general discrete series representation, the analogues of $F$ and $E_{k,N}^\chi( - , s)$ were vector-valued objects; in contrast, our formula involves only  scalar-valued functions. This is a key point of the present work.

We hope that our  pullback formula will be useful for arithmetic applications beyond what we pursue  in this paper. A particularly fruitful direction might be in the direction of the Bloch-Kato conjecture, as in recent work of Mahesh Agarwal, Jim Brown, Krzysztof Klosin, and others.
\subsection{Nearly holomorphic modular forms and  arithmeticity}
To obtain results about the algebraicity of critical $L$-values, we delve deeper into the arithmetic nature of the two smooth modular forms given above. The general arithmetic principle here is that whenever a smooth modular/automorphic form is holomorphic, or close to being holomorphic, it is likely to have useful arithmetic properties.  In this case, if $0 \le m \le \frac{k}2 - n -1$ is an integer, then we prove that $E_{k,N}^\chi(Z, -m)$ is a \emph{nearly holomorphic} Siegel modular form (of degree $2n$). The given range of $m$ corresponds to the integers for which the series defining $E_{k,N}^\chi(Z, -m)$ is absolutely convergent.\footnote{It should be possible to prove near-holomorphy for $m$ in the larger range $0 \le m \le \frac{k}2 - \frac{n+1}2$ using methods similar to that of Shimura \cite{Shimura1987, shibook2} but we do not do so for the sake of simplicity.} In fact, the function $E_{k,N}^\chi(Z, -m)$ is obtained by applying Maass-Shimura differential operators on a holomorphic Eisenstein series and consequently, a well-known theorem of Harris \cite{hareis} guarantees the arithmeticity of its Fourier coefficients.

The final step is  to prove that the inner product of $F(Z)$ and $E_{k,N}^\chi( Z , -m)$  is $\Aut(\C)$ equivariant. It is here that we are forced to assume $n=2$. In this case, our recent investigation of lowest weight modules \cite{PSS14} of $\Sp_4(\R)$ and in particular the ``structure theorem" proved therein allows us to define an $\Aut(\C)$ equivariant isotypic projection map from the space of \emph{all} nearly holomorphic modular forms to the subspace of cusp forms corresponding to a particular infinity-type. Once this is known, Theorem \ref{t:mainintro} follows by a standard linear algebra argument going back at least to Garrett \cite{gar2}.

It is worth contrasting our approach here with previously proved results on the critical $L$-values of holomorphic vector-valued Siegel cusp forms such as the result of Kozima \cite{kozima} mentioned earlier. In Kozima's work, the modular forms involved in the integral representation are vector-valued and the cusp form holomorphic; ours involves two scalar-valued modular forms that are not holomorphic. Our approach allows us to incorporate everything smoothly into an adelic setup and exactly evaluate the archimedean integral. But the price we pay is that the arithmeticity of the automorphic forms is not  automatic (as we cannot immediately appeal to the arithmetic geometry inherent in holomorphic modular forms). In particular, this is the reason we are forced to restrict ourselves to $n=2$ in the final section of this paper, where we prove Theorem \ref{t:mainintro}.

We expect that an analogue of the structure theorem for nearly holomorphic forms proved in \cite{PSS14} for $n=2$ should continue to hold for general $n$. This would immediately lead to an extension of Theorem \ref{t:mainintro} for any $n$. It would also be worth extending Theorem \ref{t:mainintro} to the full set of critical points in the right half plane (currently, it excludes the one or two critical points in the range $1 \le r \le 3$) by taking a closer look at the arithmeticity of the nearly holomorphic Eisenstein series outside the region of absolute convergence. We hope to come back to some of these problems in the future.

\section{Preliminaries}
\subsection{Basic notations and definitions}\label{basicnotdefsec}
Let $F$ be an algebraic number field and $\A$ the ring of adeles of $F$. For a positive integer $n$ let $G_{2n}$ be the algebraic $F$-group $\GSp_{2n}$, whose $F$-points are given by
$$
 G_{2n}(F)=\{g\in\GL(2n,F)\,:\,^tgJ_ng=\mu_n(g)J_n,\:\mu_n(g)\in F^\times\},\qquad J_n=\mat{}{I_n}{-I_n}{}.
$$
The symplectic group $\Sp_{2n}$ consists of those elements $g\in G_{2n}$ for which the multiplier $\mu_n(g)$ is $1$. Let $P_{2n}$ be the Siegel parabolic subgroup of $G_{2n}$, consisting of matrices whose lower left $n\times n$-block is zero. Let $\delta_{P_{2n}}$ be the modulus character of $P_{2n}(\A)$. It is given by
\begin{equation}\label{modulus-char}
 \delta_{P_{2n}}(\mat{A}{X}{}{v\,^t\!A^{-1}}) = |v^{-\frac{n(n+1)}2}\det(A)^{n+1}|, \qquad \text{ where } A \in \GL_n(\A),\: v \in \GL_1(\A),
\end{equation}
and $|\cdot|$ denotes the global absolute value, normalized in the standard way.

Fix the following embedding of $H_{2a,2b} := \{(g, g') \in G_{2a} \times G_{2b}: \mu_a(g) = \mu_b(g') \}$ in $G_{2a+2b}$:
\begin{equation}\label{embedding-defn}
H_{2a,2b} \ni (\mat{A_1}{B_1}{C_1}{D_1}, \mat{A_2}{B_2}{C_2}{D_2}) \longmapsto \left[\begin{smallmatrix}A_1&&-B_1\\&A_2&&B_2\\-C_1&&D_1\\&C_2&&D_2\end{smallmatrix}\right] \in \GSp_{2a+2b}.
\end{equation}
We will also let $H_{2a,2b}$ denote its image in $G_{2a+2b}$.

Let $G$ be any reductive algebraic group defined over $F$. For a place $v$ of $F$ let $(\pi_v,V_v)$ be an \emph{admissible representation} of $G(F_v)$. If $v$ is non-archimedean, then this means that every vector in $V_v$ is smooth, and that for every open-compact subgroup $\Gamma$ of $G(F_v)$ the space of fixed vectors $V_v^\Gamma$ is finite-dimensional. If $v$ is archimedean, then it means that $V_v$ is an admissible $(\mathfrak{g},K)$-module, where $\mathfrak{g}$ is the Lie algebra of $G(F_v)$ and $K$ is a maximal compact subgroup of $G(F_v)$. We say that $\pi_v$ is \emph{unitary} if there exists a $G(F_v)$-invariant (non-archimedean case) resp.\ $\mathfrak{g}$-invariant (archimedean case) hermitian inner product on $V_v$. In this case, and assuming that $\pi_v$ is irreducible, we can complete $V_v$ to a unitary Hilbert space representation $\bar V_v$, which is unique up to unitary isomorphism. We can recover $V_v$ as the subspace of $\bar V_v$ consisting of smooth (non-archimedean case) resp.\ $K$-finite (archimedean case) vectors.

We define automorphic representations as in \cite{BorelJacquet1979}. In particular, when we say ``automorphic representation of $G(\A)$'', we understand that at the archimedean places we do not have actions of $G(F_v)$, but of the corresponding $(\mathfrak{g},K)$-modules. All automorphic representations are assumed to be irreducible. Cuspidal automorphic representations are assumed to be unitary. Any such representation $\pi$ is isomorphic to a restricted tensor product $\otimes\pi_v$, where $\pi_v$ is an irreducible, admissible, unitary representation of $G(F_v)$.

For a place $v$ of $F$, let $\sigma,\chi_1,\cdots, \chi_n$ be characters of $F_v^\times$. We denote by $\chi_1\times \cdots \times \chi_n\rtimes\sigma$ the representation of $G_{2n}(F_v)$ parabolically induced from the character
\begin{equation}\label{GnBorelinducedeq}
 \left[\begin{smallmatrix}
  a_1&\cdots&*&*&\cdots&*\\&\ddots&\vdots&\vdots&&\vdots\\&&a_n&*&\cdots&*\\&&&a_0 a_1^{-1}&&\\&&&\vdots&\ddots&\\&&&*&\cdots&a_0a_n^{-1}
 \end{smallmatrix}\right]
 \longmapsto\chi_1(a_1)\cdot\ldots\cdot\chi_n(a_n)\sigma(a_0)
\end{equation}
of the standard Borel subgroup of $G_{2n}(F_v)$. Restricting all functions in the standard model of this representation to $\Sp_{2n}(F_v)$, we obtain a Borel-induced representation of $\Sp_{2n}(F_v)$ which is denoted by $\chi_1\times\ldots\times\chi_n\rtimes1$.

We also define parabolic induction from $P_{2n}(F_v)$. Let $\chi$ and $\sigma$ be characters of $F_v^\times$. Then $\chi\rtimes\sigma$ is the representation of $G_{2n}(F_v)$ parabolically induced from the character
\begin{equation}\label{GnSiegelinducedeq}
 \mat{A}{*}{}{u\,^t\!A^{-1}}\longmapsto\sigma(u)\chi(\det(A))
\end{equation}
of $P_{2n}(F_v)$. The center of $G_{2n}(F_v)$ acts on $\chi\rtimes\sigma$ via the character $\chi^n\sigma^2$. Restricting the functions in the standard model of $\chi\rtimes\sigma$ to $\Sp_{2n}(F_v)$, we obtain the Siegel-induced representation of $\Sp_{2n}(F_v)$ denoted by $\chi\rtimes1$.

We fix a Haar measure on $\Sp_{2n}(\R)$, as follows. Let $\H_n$ be the Siegel upper half space of degree $n$, consisting of all complex, symmetric $n\times n$ matrices $X+iY$ with $X,Y$ real and $Y$ positive definite. The group $\Sp_{2n}(\R)$ acts transitively on $\H_n$ in the usual way. The stabilizer of the point $I:=i1_n\in\H_n$ is the maximal compact subgroup $K=\Sp_{2n}(\R)\cap{\rm O}(2n)\cong U(n)$. We transfer the classical $\Sp_{2n}(\R)$-invariant measure on $\H_n$ to a left-invariant measure on $\Sp_{2n}(\R)/K$. We also fix a Haar measure on $K$ so that $K$ has volume $1$. The measures on $\Sp_{2n}(\R)/K$ and $K$ define a unique Haar measure on $\Sp_{2n}(\R)$. Let $F$ be a measurable function on $\Sp_{2n}(\R)$ that is right $K$-invariant. Let $f$ be the corresponding function on $\H_n$, i.e., $F(g)=f(gI)$. Then, by these definitions,
\begin{equation}\label{Ghaarpropeq1}
 \int\limits_{\Sp_{2n}(\R)}F(h)\,dh=\int\limits_{\H_n}f(Z)\det(Y)^{-(n+1)}\,dX\,dY.
\end{equation}
We shall always use the Haar measure on $\Sp_{2n}(\R)$ characterized by the property \eqref{Ghaarpropeq1}.

Haar measures on $\Sp_{2n}(F)$, where $F$ is a non-archimedean field with ring of integers $\OF$, will be fixed by requiring that the open-compact subgroup $\Sp_{2n}(\OF)$ has volume $1$. The Haar measure on an adelic group $\Sp_{2n}(\A)$ will always be taken to be the product measure of all the local measures. (At the complex places, the precise choice of measure will be irrelevant for this work, so one can make an arbitrary choice.)

\subsection{Some coset decompositions}
For $0 \le r \le n$, let $\alpha_r \in \Sp_{4n}(\Q)$ be the matrix
\begin{equation}\label{alphardefeq}
 \alpha_r = \left[\begin{smallmatrix}I_n &0&0&0 \\ 0& I_n&0&0\\ 0&\widetilde{I_r}&I_n&0\\ \widetilde{I_r}&0&0&I_n \end{smallmatrix}\right],
\end{equation}
where the $n\times n$ matrix $\widetilde{I_r}$ is given by  $\widetilde{I_r}= \mat{0_{n-r}}{0}{0}{I_r}$.

\begin{proposition}\label{doublecosetdec}
 We have the double coset decomposition
 $$
  G_{4n}(F) = \bigsqcup_{r=0}^n  P_{4n}(F)\alpha_r H_{2n,2n}(F).
 $$
\end{proposition}

\begin{proof}
This follows from Shimura~\cite[Prop.\ 2.4]{shibook1}.
\end{proof}

For our purposes, it is nicer to work with the coset representatives $Q_r := \alpha_r \cdot (I_{4n-2r}, J_r)$ where $J_r = \mat{}{I_r}{-I_r}{}$. It is not hard to see that $(I_{4n-2r}, J_r) $ is actually an element of $H_{2n,2n} $, so that
$$
 P_{4n}(F)\alpha_r H_{2n,2n}(F) =  P_{4n}(F)Q_r H_{2n,2n}(F).
$$
One can write down the matrix $Q_r$ explicitly as follows,
\begin{equation}\label{Qrdefeq}
 Q_r = \left[\begin{smallmatrix}I_n &0&0&0 \\ 0& I'_{n-r}&0&\widetilde{I_r}\\ 0&0&I_n&\widetilde{I_r}\\ \widetilde{I_r}&-\widetilde{I_r}&0&I'_{n-r} \end{smallmatrix}\right],
\end{equation}
where $I'_{n-r}=I_n-\tilde I_r=\mat{I_{n-r}}{0}{0}{0}$.

For $0 \le r \le n$, let $P_{2n, r}$ be the maximal parabolic subgroup (proper if $r \neq n$) of $G_{2n}$ consisting of matrices whose lower-left $(n+r) \times (n-r)$ block is zero. Its Levi component is isomorphic to $\GL_{n-r}\times G_r$. Note that $P_{2n,0} = P_{2n}$ and $P_{2n,n} = G_{2n}$. Let $N_{2n,r}$ denote the unipotent radical of $P_{2n, r}$.

The next lemma expresses the reason why $\{Q_r\}$ is more convenient than  $\{\alpha_r\}$ for the double coset representatives. Let $d:P_{4n}\to\GL_1$ be the homomorphism defined by $d(p)=v^{-n}\det(A)$ for $p=\mat{A}{}{}{v\,^t\!A^{-1}}n'$ with $v\in\GL_1, A\in\GL_{2n}$ and $n' \in N_{4n,0}$. Note that $\delta_{P_{4n}}(p)=|d(p)|^{2n+1}$ for $p\in P_{4n}(\A)$.
\begin{lemma}\label{lemmaqrprop}
 \begin{enumerate}
  \item Let $0 \leq r < n$ and suppose that
   $$
    p_1 =   \left[\begin{smallmatrix}g_1\\&I_r&&&\\&&\,^tg_1^{-1}\\&&&&I_r\end{smallmatrix}\right]\cdot n_1,\quad  p_2 =   \left[\begin{smallmatrix}g_2\\&I_r&&&\\&&\,^tg_2^{-1}\\&&&&I_r\end{smallmatrix}\right]\cdot n_2, \quad g_1, g_2\in\GL_{n-r}, \quad n_1, n_2\in N_{2n,r}$$      are certain elements of $P_{2n, r}$. Then the matrix $X = Q_r(p_1,p_2) Q_r^{-1}$ lies in $P_{4n}$ and satisfies $d(X)=\det(g_1)\det(g_2)$.
  \item  Let $0 \leq r < n$. Let $x_1\in\Sp_{2r}$ and $x=(1,x_1)\in\Sp_{2n}$. Then the matrix $X = Q_r(x,x) Q_r^{-1}$ lies in $P_{4n}$ and satisfies $d(X) = 1$.
  \item Let $g \in G_{2n}$. Then the matrix $X = Q_n(g,g) Q_n^{-1}$ lies in $P_{4n}$ and satisfies $d(X) = 1$.
\end{enumerate}
\end{lemma}
\begin{proof}
This follows by direct verification.
\end{proof}

Next we provide a set of coset representatives for $P_{4n}(F) \bs  P_{4n}(F)Q_r H_{2n,2n}(F)$.

\begin{proposition}\label{rightcosetdec}
 For each $0 \le r \le n$, we have the coset decomposition
 $$
  P_{4n}(F)Q_r H_{2n,2n}(F) = \bigsqcup_{\substack{x = (1, x_1) \in\Sp_{2n}(F)\\ x_1 \in \Sp_{2r}(F) }}\;\bigsqcup_{\substack{y \in P_{2n, r}(F) \bs G_{2n}(F) \\ z \in P_{2n, r}(F) \bs G_{2n}(F)}} P_{4n}(F) \cdot Q_r \cdot (xy, z),
 $$
 where we choose the representatives $y,z$ to be in $\Sp_{2n}(F)$.
\end{proposition}
\begin{proof} This follows from Shimura~\cite[Prop.\ 2.7]{shibook1}. Note that the choice of representatives $y$ and $z$ is irrelevant by i) and ii) of Lemma \ref{lemmaqrprop}.
\end{proof}
\section{Eisenstein series and zeta integrals}
\subsection{Degenerate principal series representations}
Let $\chi$ be a character of $F^\times \backslash \A^\times$. We define a character on $P_{4n}(\A)$, also denoted by $\chi$, by $\chi(p) = \chi(d(p))$. For a complex number $s$, let
\begin{equation}\label{ind-repn}
 I(\chi,s) = {\rm Ind}^{G_{4n}(\A)}_{P_{4n}(\A)} ( \chi \delta_{P_{4n}}^s).
\end{equation}
Thus, $f(\,\cdot\,, s) \in I(\chi,s)$ is a smooth complex-valued function satisfying
\begin{equation}\label{ind-repn-fctn}
 f(pg, s)  = \chi(p) \delta_{P_{4n}}(p)^{s + \frac12} f(g,s)
\end{equation}
for all $p \in P_{4n}(\A)$ and $g \in G_{4n}(\A)$. Note that these functions are invariant under the center of $G_{4n}(\A)$. Let $I_v(\chi_v,s)$ be the analogously defined local representation. Using the notation introduced in Sect.~\ref{basicnotdefsec}, we have
\begin{equation}\label{Ichisaltnoteq}
 I_v(\chi_v,s)=\chi_v|\cdot|_v^{(2n+1)s}\,\rtimes\,\chi_v^{-n}|\cdot|_v^{-n(2n+1)s}.
\end{equation}
We have $I(\chi,s)\cong\otimes I_v(\chi_v,s)$ in a natural way. Given $f\in I(\chi,s)$, it follows from iii) of Lemma \ref{lemmaqrprop} that, for all $h\in\Sp_{2n}(\A)$,
\begin{equation}\label{Qninveq1}
 f(Q_n \cdot (gh, g), s)=f(Q_n \cdot (h, 1), s)\qquad\text{for }g\in G_{2n}(\A).
\end{equation}
We will mostly use this observation in the following form. Let $f_v\in I_v(\chi_v,s)$ and $K$ a maximal compact subgroup of $\Sp_{2n}(F_v)$. Then
\begin{equation}\label{Qninveq2}
 f_v(Q_n \cdot (k_1hk_2, 1), s)=f_v(Q_n \cdot (h, 1)(k_2,k_1^{-1}), s)\qquad\text{for }k_1,k_2\in K
\end{equation}
and all $h\in\Sp_{2n}(F_v)$.

In preparation for the next result, and for the unramified calculation in Sect.~\ref{unramifiedcalcsec}, we recall some facts concerning the unramified Hecke algebra at a non-archimedean place $v$ of $F$. We fix a uniformizer $\varpi$ of $F_v$. Let $K=\Sp_{2n}(\OF_v)$. Recall the Cartan decomposition
\begin{align}
 \label{cartaneq2c}\Sp_{2n}(F)&=\bigsqcup_{\substack{e_1,\cdots, e_n\in\Z\\0\leq e_1\leq \cdots \leq e_n}}K_{e_1,\cdots, e_n},
\end{align}
where
\begin{align}
 \label{cartaneq2d}K_{e_1,\cdots,e_n}&=K\,{\rm diag}(\varpi^{e_1}, \cdots, \varpi^{e_n},\varpi^{-e_1},\cdots,\varpi^{-e_n}) K.
\end{align}
Consider the spherical Hecke algebra $\mathcal{H}_n$ consisting of left and right $K$-invariant functions on $\Sp_{2n}(F)$. The structure of this Hecke algebra is described by the Satake isomorphism
\begin{equation}\label{satakeisoeq}
 \mathcal{S}:\:\mathcal{H}_n\longrightarrow\C[X_1^{\pm1}, \cdots, X_n^{\pm1}]^W,
\end{equation}
where the superscript $W$ indicates polynomials that are invariant under the action of the Weyl group of $\Sp_{2n}$. Let $T_{e_1,\ldots,e_n}$ be the characteristic function of the set $K_{e_1,\ldots,e_n}$ defined in (\ref{cartaneq2d}). Then $T_{e_1,\ldots,e_n}$ is an element of the Hecke algebra $\mathcal{H}_n$. The values $\mathcal{S}(T_{e_1,\ldots,e_n})$ are encoded in the \emph{rationality theorem}
\begin{equation}\label{bocherersformulaeq}
 \sum_{\substack{e_1,\ldots,e_n\in\Z\\0\leq e_1\leq\ldots\leq e_n}}\mathcal{S}(T_{e_1,\ldots,e_n})Y^{e_1+\ldots+e_n}=\frac{1-Y}{1-q^nY}\prod_{i=1}^n\frac{1-q^{2i}Y^2}{(1-X_iq^nY)(1-X_i^{-1}q^nY)}.
\end{equation}
This identity of formal power series is the main result of \cite{boch1986}.

Let $\chi_1,\ldots,\chi_n$ be unramified characters of $F_v^\times$. Let $\pi$ be the unramified constituent of $\chi_1\times\ldots\times\chi_n\rtimes1$. The numbers $\alpha_i=\chi_i(\varpi)$, $i=1,\ldots,n$, are called the \emph{Satake parameters} of $\pi$. Let $v_0$ be a spherical vector in $\pi$. It is unique up to scalars. Hence, if we act on $v_0$ by an element $T$ of $\mathcal{H}_n$, we obtain a multiple of $v_0$. This multiple is given by evaluating $\mathcal{S}T$ at the Satake parameters, i.e.,
\begin{equation}\label{Heckesatakeactioneq}
 \int\limits_{\Sp_{2n}(F)}T(x)\pi(x)v_0\,dx=(\mathcal{S}T)(\alpha_1,\ldots,\alpha_n)v_0.
\end{equation}

Now assume that $v$ is a real place. Let $K=\Sp_{2n}(\R)\cap{\rm O}(2n)\cong U(n)$ be the standard maximal compact subgroup of $\Sp_{2n}(\R)$. Let $\mathfrak{g}$ be the (uncomplexified) Lie algebra of $\Sp_{2n}(\R)$, and let $\mathfrak{a}$ be the subalgebra consisting of diagonal matrices. Let $\Sigma$ be the set of restricted roots with respect to $\mathfrak{a}$. If $e_i$ is the linear map sending ${\rm diag}(a_1,\ldots,a_n,-a_1,\ldots,-a_n)$ to $a_i$, then $\Sigma$ consists of all $\pm(e_i-e_j)$ for $1\leq i<j\leq n$ and $\pm(e_i+e_j)$ for $1\leq i\leq j\leq n$. As a positive system we choose
\begin{equation}\label{posrestrictedrootseq}
 \Sigma^+=\{e_j-e_i:1\leq i<j\leq n\}\,\cup\,\{-e_i-e_j:1\leq i\leq j\leq n\}
\end{equation}
(what is more often called a \emph{negative} system). Then the positive Weyl chamber is
\begin{equation}\label{posweylchambereq}
 \mathfrak{a}^+=\{{\rm diag}(a_1,\ldots,a_n,-a_1,\ldots,-a_n):a_1<\ldots<a_n<0\}.
\end{equation}
By Proposition 5.28 of \cite{knapp1986}, or Theorem 5.8 in Sect.~I.5 of \cite{Helgason1984}, we have the integration formula
\begin{equation}\label{KAKintegrationeq}
 \int\limits_{\Sp_{2n}(\R)}\phi(h)\,dh=\alpha_n\int\limits_{K}\int\limits_{\mathfrak{a}^+}\int\limits_{K}\bigg(\prod_{\lambda\in\Sigma^+}\sinh(\lambda(H))\bigg)\phi(k_1\exp(H)k_2)\,dk_1\,dH\,dk_2,
\end{equation}
which we will use for continuous, non-negative functions $\phi$ on $\Sp_{2n}(\R)$. The measure $dH$ in \eqref{KAKintegrationeq} is the Lebesgue measure on $\mathfrak{a}^+\subset\R^n$. The positive constant $\alpha_n$ relates the Haar measure given by the integration on the right hand side to the Haar measure $dh$ on $\Sp_{2n}(\R)$ we fixed once and for all by \eqref{Ghaarpropeq1}. We will calculate $\alpha_n$ explicitly in Appendix \ref{KAKmeasureapp}.
\subsection{Local zeta integrals}\label{convsec}
\begin{lemma}\label{convergencelemma}
 \begin{enumerate}
  \item Let $v$ be any place of $F$ and let $f_v\in I_v(\chi_v,s)$. Then, for ${\rm Re}(s)$ large enough, the function on $\Sp_{2n}(F_v)$ defined by $h\mapsto f_v(Q_n\cdot(h,1),s)$ is in $L^1(\Sp_{2n}(F_v))$.
  \item Let $f\in I(\chi,s)$. Then, for ${\rm Re}(s)$ large enough, the function on $\Sp_{2n}(\A)$ defined by $h\mapsto f(Q_n\cdot(h,1),s)$ is in $L^1(\Sp_{2n}(\A))$.
 \end{enumerate}
\end{lemma}
\begin{proof}
Since ii) follows from i) by definition of the adelic measure, we only have to prove the local statement. To ease notation, we will omit all subindices $v$. Define a function $f'(g,s)$ by
\begin{equation}\label{convergencelemmaeq1}
 f'(g,s)=\int\limits_{K}\int\limits_{K}|f(g(k_2,k_1^{-1}), s)|\,dk_1\,dk_2,\qquad g\in G_{4n}(F).
\end{equation}
From \eqref{ind-repn-fctn}, we see that
\begin{equation}\label{ind-repn-fctn2}
 f'(pg, s)  = |\chi(p) \delta_{P_{4n}}(p)^{s + \frac12}|\,f'(g,s)
\end{equation}
for all $p \in P_{4n}(F)$ and $g \in G_{4n}(F)$. Equation \eqref{Qninveq2} implies that
\begin{equation}\label{convergencelemmaeq2}
 \int\limits_{\Sp_{2n}(F)}|f(Q_n \cdot (h, 1), s)|\,dh=\int\limits_{\Sp_{2n}(F)}|f'(Q_n \cdot (h, 1), s)|\,dh.
\end{equation}
From now on we will assume that $v$ is a non-archimedean place; the proofs for the real and complex places are similar. It follows from \eqref{cartaneq2c} that
\begin{align}\label{convergencelemmaeq3}
 &\int\limits_{\Sp_{2n}(F)}| f'(Q_n \cdot (h, 1), s)|\,dh=\sum_{\substack{e_1,\cdots, e_n\in\Z\\0\leq e_1\leq \cdots \leq e_n}}\;\:\int\limits_{K_{e_1,\cdots, e_n}}|f'(Q_n \cdot (h, 1), s)|\,dh\nonumber\\
 &\hspace{5ex}=\sum_{\substack{e_1,\cdots, e_n\in\Z\\0\leq e_1\leq \cdots \leq e_n}}{\rm vol}(K_{e_1,\cdots, e_n})\:\big|f'(Q_n \cdot ({\rm diag}(\varpi^{e_1}, \cdots, \varpi^{e_n},\varpi^{-e_1},\cdots,\varpi^{-e_n}), 1), s)\big|.
\end{align}
From \eqref{ind-repn-fctn2}, we find
\begin{align}\label{convergencelemmaeq4}
 &f'(Q_n \cdot({\rm diag}(\varpi^{e_1}, \cdots, \varpi^{e_n},\varpi^{-e_1}, \cdots,\varpi^{-e_n}), 1), s)\nonumber\\
 &\hspace{8ex}=\Big(|\chi(\varpi)|q^{-(2n+1)({\rm Re}(s)+1/2)}\Big)^{e_1+\cdots+e_n}f'(\left[\begin{smallmatrix}I_n\\&I_n\\&A&I_n\\A&&&I_n\end{smallmatrix}\right]\left[\begin{smallmatrix}I_n\\&&&I_n\\&&I_n\\&-I_n\end{smallmatrix}\right],s)
\end{align}
with $A={\rm diag}(\varpi^{e_1},\ldots,\varpi^{e_n})$. By smoothness, the term $f'(\ldots)$ in the second line of \eqref{convergencelemmaeq4} takes only finitely many values, and can therefore be estimated by a constant $C$ independent of $e_1,\ldots,e_n$. Thus
\begin{equation}\label{convergencelemmaeq5}
 \int\limits_{\Sp_{2n}(F)}| f'(Q_n \cdot (h, 1), s)|\,dh\leq C\sum_{\substack{e_1,\cdots, e_n\in\Z\\0\leq e_1\leq \cdots \leq e_n}}{\rm vol}(K_{e_1,\cdots, e_n})\:c(s)^{e_1+\cdots+e_n},
\end{equation}
where $c(s)=|\chi(\varpi)|\,q^{-(2n+1)({\rm Re}(s)+1/2)}$. Consider the trivial representation $1_{\Sp_{2n}(F)}$ of $\Sp_{2n}(F)$. Since it is a subrepresentation of $|\cdot|^{-n}\times\ldots\times|\cdot|^{-1}$, its Satake parameters are $\alpha_i=q^i$ for $i=1,\ldots,n$. Let $v_0$ be a spanning vector of $1_{\Sp_{2n}(F)}$. Then
$T_{e_1,\ldots,e_n}v_0={\rm vol}(K_{e_1,\cdots, e_n})v_0.$ By \eqref{Heckesatakeactioneq} it follows that ${\rm vol}(K_{e_1,\cdots, e_n})=\mathcal{S}(T_{e_1,\ldots,e_n})(\alpha_1,\ldots,\alpha_n)$, where $\alpha_i=q^i$. Substituting $\alpha_i=q^i$ into \eqref{bocherersformulaeq}, we get
\begin{equation}\label{convergencelemmaeq7}
 \sum_{\substack{0\leq e_1\leq\ldots\leq e_n}}{\rm vol}(K_{e_1,\cdots, e_n})Y^{e_1+\ldots+e_n}= \frac{1-Y}{1-q^nY} \prod_{i=1}^n\frac{1+q^iY}{1-q^{n+i}Y},
\end{equation}
an identity of formal power series in $Y$. We see that \eqref{convergencelemmaeq5} is convergent if $c(s)$ is sufficiently small, i.e., if ${\rm Re}(s)$ is sufficiently large.
\end{proof}

\begin{proposition}\label{interopprop}
 Let $\chi$ be a character of $F^\times\backslash\A^\times$.
 \begin{enumerate}
  \item For a place $v$ of $F$ let $f_v\in I_v(\chi_v,s)$. Let $(\pi_v,V_v)$ be an admissible, unitary representation of $\Sp_{2n}(F_v)$. If ${\rm Re}(s)$ is sufficiently large, then the integral
   \begin{equation}\label{Rsflocaleq}
    Z_v(s;f_v, w_v):=\int\limits_{\Sp_{2n}(F_v)} f_v(Q_n \cdot (h, 1), s) \pi_v(h)w_v\,dh
   \end{equation}
   converges absolutely to an element of $\bar V_v$, for any $w_v$ in the Hilbert space completion $\bar V_v$. If $w_v$ is in $V_v$, then so is $Z_v(s;f_v, w_v)$.
  \item Let $\pi\cong\otimes\pi_v$ be a cuspidal, automorphic representation of $G_{2n}(\A)$. Let $V$ be the space of automorphic forms on which $\pi$ acts. If ${\rm Re}(s)$ is sufficiently large, then the function
   \begin{equation}\label{interoppropeq2}
    g\longmapsto\int\limits_{\Sp_{2n}(\A)} f(Q_n \cdot (h, 1), s)\phi(gh)\,dh
   \end{equation}
   is an element of $V$, for every $\phi\in V$. If $f=\otimes f_v$ with $f_v\in I_v(\chi_v,s)$, and if $\phi$ corresponds to the pure tensor $\otimes w_v$, then the function \eqref{interoppropeq2} corresponds to the pure tensor $\otimes Z_v(s;f_v, w_v)$.
 \end{enumerate}
\end{proposition}
\begin{proof}
i) The absolute convergence follows form Lemma \ref{convergencelemma} i). The second assertion can be verified in a straightforward way using \eqref{Qninveq2}.

ii) Lemma \ref{convergencelemma} ii) implies that the integral
\begin{equation}\label{interoppropeq3}
 \int\limits_{\Sp_{2n}(\A)}  f(Q_n \cdot(h  ,1), s) (R(h)\phi)\,dh,
\end{equation}
where $R$ denotes right translation, converges absolutely to an element in the Hilbert space completion $\bar V$ of $V$. With the same argument as in the local case we see that this element has the required smoothness properties that make it an automorphic form, thus putting it into $V$. Evaluating at $g$, we obtain the first assertion. The second assertion follows by applying a unitary isomorphism $\pi\cong\otimes\pi_v$ to \eqref{interoppropeq3}.
\end{proof}
\subsection{The basic identity}
Let $I(\chi,s)$ be as in \eqref{ind-repn}, and let $f(\cdot,s)$ be a section whose restriction to the standard maximal compact subgroup of $G_{4n}(\A)$ is independent of $s$. Consider the Eisenstein series on $G_{4n}(\A)$ which, for $\Re(s)>\frac12$, is given by the absolutely convergent series
\begin{equation}\label{Eis-ser-defn}
 E(g,s,f) = \sum\limits_{\gamma \in P_{4n}(F) \backslash G_{4n}(F)} f(\gamma g, s),
\end{equation}
and defined by analytic continuation outside this region. Let $\pi$ be a cuspidal automorphic representation of $G_{2n}(\A)$. Let $V_\pi$ be the space of cuspidal automorphic forms realizing $\pi$. For any automorphic form $\phi $ in $V_\pi$ and any $s\in \C$ define a function $Z(s;f, \phi)$ on $G_{2n}(F) \bs G_{2n}(\A)$ by
\begin{equation}\label{maindefinition}
 Z(s; f, \phi) (g) = \int\limits_{\Sp_{2n}(F) \bs\,g\cdot\Sp_{2n}(\A)} E((
 h  ,g), s, f) \phi(h)\,dh.
\end{equation}

\begin{remark}Note that  $g\cdot\Sp_{2n}(\A) = \{h \in G_{2n}(\A):\mu_n(h) = \mu_n(g)\}.$ \end{remark}

 \begin{remark} $E(g,s,f)$ is slowly increasing away from its poles and $\phi$ is rapidly decreasing. So   $Z(s; f, \phi)(g)$ converges  absolutely for $s \in \C$ away from the poles of the Eisenstein series and defines an automorphic form on $G_{2n}$. We will see soon that $Z(s;f, \phi)$ in fact belongs to $V_\pi$.
  \end{remark}Using Propositions~\ref{doublecosetdec} and~\ref{rightcosetdec}, we conclude that
$
 Z(s; f, \phi)(g) = \sum_{r=0}^n Z_r(s;f, \phi)(g),
$
where
$$
 Z_r(s;f, \phi)(g) = \int\limits_{\Sp_{2n}(F) \bs \, g\cdot\Sp_{2n}(\A)}\,\sum_{\substack{x = (1, x_1) \in \Sp_{2n}(F)\\ x_1 \in \Sp_{2r}(F) }}\;\sum_{\substack{y \in P_{2n, r}(F) \bs G_{2n}(F) \\ z \in P_{2n, r}(F) \bs G_{2n}(F)}}f(Q_r \cdot(xyh,zg), s) \phi(h)\,dh.
$$

\begin{proposition}\label{vanishingprop}Let $0 \le r <n$. Then $Z_r(s;f, \phi)(g) = 0.$
\end{proposition}

\begin{proof} Let $N_{2n,r}$ be the unipotent radical of the maximal parabolic subgroup $P_{2n, r}$. Let $P'_{2n, r}=P_{2n, r}\cap\Sp_{2n}$. Then
\begin{align*}
 Z_r(s;f, \phi)(g) &= \int\limits_{\Sp_{2n}(F) \bs \, g\cdot\Sp_{2n}(\A)}\;\sum_{\substack{x = (1, x_1) \in \Sp_{2n}(F)\\ x_1 \in \Sp_{2r}(F) }}\;\sum_{\substack{y \in P'_{2n, r}(F) \bs \Sp_{2n}(F) \\ z \in P_{2n, r}(F) \bs G_{2n}(F)}}f(Q_r \cdot(xy h  ,zg), s) \phi(h)\,dh\\
 &= \int\limits_{P'_{2n, r}(F) \bs \, g\cdot\Sp_{2n}(\A)}\;\sum_{\substack{x = (1, x_1) \in \Sp_{2n}(F)\\ x_1 \in \Sp_{2r}(F) }}\;\sum_{z \in P_{2n, r}(F) \bs G_{2n}(F)}f(Q_r \cdot(xh  ,zg), s) \phi(h)\,dh\\
 &= \int\limits_{N_{2n,r}(\A)P'_{2n, r}(F) \bs \,g\cdot\Sp_{2n}(\A)}\;\int\limits_{N_{2n,r}(F) \bs N_{2n,r}(\A)}  \sum_{x,z}f(Q_r \cdot(xnh  ,zg), s) \phi(nh)\,dn\,dh.
\end{align*}
Since the element $x$ normalizes $N_{2n,r}(\A)$ and $N_{2n,r}(F)$, we can commute $x$ and $n$. Then $n$ can be omitted by i) of Lemma \ref{lemmaqrprop}. Hence
$$
 Z_r(s;f, \phi)(g)= \int\limits_{N_{2n,r}(\A)P'_{2n, r}(F) \bs \,g\cdot\Sp_{2n}(\A)} \;\int\limits_{N_{2n,r}(F) \bs N_{2n,r}(\A)}\bigg(\sum_{x,z}f(Q_r \cdot(xh  ,zg), s)\bigg)\phi(nh)\,dn\,dh,
$$
and the cuspidality of $\phi$ implies that this is zero.
\end{proof}

For the following theorem, which is the main result of this section, recall the local integrals defined in \eqref{Rsflocaleq}.

\begin{theorem}[Basic identity]\label{basicidentity}
 Let $\phi \in V_\pi$ be a cusp form which corresponds to a pure tensor $\otimes_v \phi_v$ via the isomorphism $\pi\cong\otimes\pi_v$. Assume that $f\in I(\chi,s)$ factors as $\otimes f_v$ with $f_v\in I(\chi_v,s)$. Let the function $Z(s;f, \phi)$ on $G_{2n}(F)\bs G_{2n}(\A)$ be defined as in~\eqref{maindefinition}. Then $Z(s;f, \phi)$ also belongs to $V_\pi$ and corresponds to the pure tensor $\otimes_v Z_v(s;f_v, \phi_v)$.
\end{theorem}
\begin{proof}
By Proposition~\ref{vanishingprop}, we have that $Z(s;f, \phi)(g) = Z_n(s;f, \phi)(g).$  But
\begin{align*}
 Z_n(s;f, \phi)(g) &= \int\limits_{\Sp_{2n}(F) \bs \, g\cdot\Sp_{2n}(\A)}  \sum_{x  \in \Sp_{2n}(F) } f(Q_n \cdot(x h  ,g), s) \phi(h)\,dh \\
 &= \int\limits_{ g\cdot\Sp_{2n}(\A)}  f(Q_n \cdot(
 h  ,g), s) \phi(h)\,dh \\
 &= \int\limits_{\Sp_{2n}(\A)}  f(Q_n \cdot(h  ,1), s) \phi(gh)\,dh,
\end{align*}
where the last step follows from \eqref{Qninveq1}. The theorem now follows from Proposition \ref{interopprop}.
\end{proof}

Our goal will be to choose, at all places, the vectors $\phi_v$ and the sections $f_v$ in such a way that $Z_v(s,f_v,\phi_v)=c_v(s)\phi_v$ for an explicitly computable function $c_v(s)$.
\section{The local integral at finite places}\label{unramifiedcalcsec}
In this section we define suitable local sections and calculate the local integrals \eqref{Rsflocaleq} for all finite places $v$. We will drop the subscript $v$ throughout. Hence, let $F$ be a non-archimedean local field of characteristic zero. Let $\OF$ be its ring of integers, $\varpi$ a uniformizer, and $\p=\varpi\OF$ the maximal ideal.

\subsection{Unramified places}
We begin with the unramified case.
Let $\chi$ be an unramified character of $F^\times$, and let $\pi$ be a spherical representation of $G_{2n}(F)$. Let $f\in I(\chi,s)$ be the normalized unramified vector, i.e., $f:\:G_{4n}\rightarrow\C$ is given by
\begin{equation}\label{sphericalsectioneq}
 f(\mat{A}{*}{}{u\,^t\!A^{-1}}k)=\chi(u^{-n}\det(A))\big|u^{-n}\det(A)\big|^{(2n+1)(s+1/2)}
\end{equation}
for $A\in\GL_{2n}(F)$, $u\in F^\times$ and $k\in G_{4n}(\OF)$. Let $v_0$ be a spherical vector in $\pi$. We wish to calculate the local integral
\begin{equation}\label{Rsflocaleq2}
 Z(s;f,v_0)=\int\limits_{\Sp_{2n}(F)} f(Q_n \cdot (h, 1), s) \pi(h)v_0 \, dh.
\end{equation}

Let $\sigma,\chi_1,\cdots, \chi_n$ be unramified characters of $F^\times$ such that $\pi$ is the unique spherical constituent of $\chi_1\times \cdots \times \chi_n\rtimes\sigma$. Let $\alpha_i=\chi_i(\varpi)$, $i=1,\ldots,n$, and $\alpha_0=\sigma(\varpi)$ be the Satake parameters of $\pi$. The dual group of $\GSp_{2n} \times \GL_1$ is $\GSpin_{2n+1}(\C) \times \GL_1(\C)$. There is a natural map $\varrho_{2n+1}: \GSpin_{2n+1}(\C) \rightarrow \SO_{2n+1}(\C)\rightarrow \GL_{2n+1}(\C)$. Consequently we get a tensor product map from $\GSpin_{2n+1}(\C) \times \GL_1(\C)$ into $\GL_{2n+1}(\C)$ which we denote also by $\varrho_{2n+1}$. The $L$-function $L(s,\pi \boxtimes \chi, \varrho_{2n+1})$ is then given as follows,
\begin{equation}\label{Lspichidefeq}
 L(s,\pi \boxtimes \chi, \varrho_{2n+1})=\frac1{1-\chi(\varpi)q^{-s}}
 \prod_{i=1}^n\frac1{(1-\chi(\varpi)\alpha_iq^{-s})
 (1-\chi(\varpi)\alpha_i^{-1}q^{-s})}.
\end{equation}
 We also define $L(s, \chi)=(1-\chi(\varpi)q^{-s})^{-1}$, as usual.
\begin{proposition}\label{unramifiedcalculationprop}
 Using the above notations and hypotheses, the local integral \eqref{Rsflocaleq2} is given by
 $$
  Z(s;f,v_0)=\frac{L((2n+1)s+1/2,\pi \boxtimes \chi, \varrho_{2n+1})}{L((2n+1)(s+1/2), \chi)\prod\limits_{i=1}^nL((2n+1)(2s+1)-2i, \chi^2)}\,v_0
 $$
 for real part of $s$ large enough.
\end{proposition}
\begin{proof}
The calculation is similar to that in the proof of Lemma \ref{convergencelemma}. By \eqref{cartaneq2c}, \eqref{Qninveq2} and \eqref{convergencelemmaeq4},
\begin{align*}
 Z(s;f,v_0)&=\sum_{\substack{e_1, \cdots, e_n\in\Z\\0\leq e_1\leq \cdots \leq e_n}}f(Q_n \cdot {\rm diag}(\varpi^{e_1}, \cdots, \varpi^{e_n},\varpi^{-e_1}, \cdots,\varpi^{-e_n}), 1), s)\int\limits_{K_{e_1, \cdots, e_n}}\pi(h)v_0 \, dh\\
  &=\sum_{\substack{e_1, \cdots, e_n\in\Z\\0\leq e_1\leq \cdots \leq e_n}}\Big(\chi(\varpi)q^{-(2n+1)(s+1/2)}\Big)^{e_1+\cdots+e_n}\,T_{e_1,\ldots,e_n}v_0,
\end{align*}
where $T_{e_1,\ldots,e_n}v_0$ is the Hecke operator introduced in Sect.~\ref{convsec}. The restriction of $\pi$ to $\Sp_{2n}(F)$ equals the spherical constituent of $\chi_1\times\ldots\times\chi_n\rtimes1$, which has Satake parameters $\alpha_i=\chi_i(\varpi)$. The assertion now follows from \eqref{Heckesatakeactioneq} and \eqref{bocherersformulaeq}.
\end{proof}
\subsection{Ramified places}\label{s:badplaces}
Now we deal with the ramified places. For a non-negative integer $m$, let $\Gamma_{2n}(\p^m)= \{g \in \Sp_{2n}(\OF) : \ g \equiv I_{2n}\bmod \p^m\}$.
\begin{lemma}\label{keylemmabadplaces}
 Let $m$ be a positive integer. Let $p \in P_{4n}(F)$ and $h \in \Sp_{2n}(F)$ be such that $Q_n^{-1} p\,Q_n (h, 1) \in \Gamma_{4n}(\p^m)$. Then $h \in \Gamma_{2n}(\p^{m})$ and $p\in P_{4n}(F)\cap\Gamma_{4n}(\p^m)$.
\end{lemma}
\begin{proof}
Since $Q_n$ normalizes $\Gamma_{4n}(\p^m)$, the hypothesis is equivalent to $p\,Q_n (h, 1) Q_n^{-1}\in \Gamma_{4n}(\p^{m})$. Write $p=\mat{*}{*}{}{g}$ with $g=\mat{g_1}{g_2}{g_3}{g_4}\in\GL_{2n}(F)$ and $h=\mat{A}{B}{C}{D}$. Then
\begin{equation}\label{keylemmabadplaceseq1}
 \left[\begin{smallmatrix}*&*&*&*\\ *&*&*&*\\&&g_1&g_2\\&&g_3&g_4\end{smallmatrix}\right]\left[\begin{smallmatrix}A&&-B\\&I_n\\-C\:&\,I_n-D\,&\:D\\A-I_n&B&-B&\:I_n\end{smallmatrix}\right]\in\Gamma_{4n}(\p^m).
\end{equation}
From the last two rows and last three columns we get \begin{alignat}{3}\label{keylemmabadplaceseq2}
 g_1(I_n-D)+g_2B&\in M_n(\p^m),&\qquad g_1D-g_2B&\in I_n+M_n(\p^m),&\qquad g_2&\in M_n(\p^m),\\
 g_3(I_n-D)+g_4B&\in M_n(\p^m),&g_3D-g_4B&\in M_n(\p^m),&g_4&\in I_n+M_n(\p^m),
\end{alignat}
where $M_n(\p^m)$ is the set of $n\times n$ matrices with entries in $\p^m$. Hence $g\in I_{2n}+M_{2n}(\p^m)=\Gamma_{2n}(\p^m)$. Multiplying \eqref{keylemmabadplaceseq1} from the left by $p^{-1}$
and looking at the lower left block, we see that $\mat{-C}{I_n-D}{A-I_n}{B}\in M_{2n}(\p^m)$, i.e., $h\in\Gamma_{2n}(\p^m)$. Then it follows from \eqref{keylemmabadplaceseq1} that $p\in\Gamma_{4n}(\p^m)$.
\end{proof}

Let $m$ be a \emph{positive}  integer such that $\chi |_{(1+\p^m)\cap \OF^\times} = 1$. Let $f(g,s)$ be the unique function on $G_{4n}(F) \times \C$ such that
\begin{enumerate}
 \item $f(pg, s)  = \chi(p) \delta_{P_{4n}}(p)^{s + \frac12} f(g,s)$ for all $g\in G_{4n}(F)$ and $p\in P_{4n}(F)$.
 \item $f(gk, s) = f(g, s)$ for all $g \in G_{4n}(F)$ and $k \in \Gamma_{4n}(\p^m)$.
 \item $f(Q_n, s) = 1$.
 \item $f(g, s) = 0$ if $g \notin  P_{4n}(F)Q_n\Gamma_{4n}(\p^m)$.
\end{enumerate}
It is easy to see that $f$ is well-defined. Evidently, $f\in I(\chi,s)$.

\begin{proposition}\label{prop:badplaces}
 Let $\pi$ be any irreducible admissible representation of $G_{2n}(F)$. Let $m$ be a positive integer such that $\chi |_{(1+\p^m)\cap \OF^\times} = 1$ and such that there exists a vector $\phi$ in $\pi$ fixed by $\Gamma_{2n}(\p^m)$. Let $f\in I(\chi,s)$ be as above. Then $Z(s;f, \phi) = r_m \phi$, where $r_m=\mathrm{vol}(\Gamma_{2n}(\p^m))$ is a non-zero rational number depending only on $f$ and $m$.
\end{proposition}
\begin{proof}
By Lemma \ref{keylemmabadplaces},
\begin{equation}
 Z(s;f, \phi)=\int\limits_{\Sp_{2n}(F)} f(Q_n \cdot (h, 1), s) \pi(h)\phi\,dh=\int\limits_{\Gamma_{2n}(\p^m)} f(Q_n \cdot (h, 1), s) \pi(h)\phi\,dh.
\end{equation}
This last integral equals $\text{vol}(\Gamma_{2n}(\p^m))\phi$, completing the proof.
\end{proof}

\section{The local integral at a real place}\label{holdiscsec}
\subsection{Holomorphic discrete series representations}
We provide some preliminaries on holomorphic discrete series representations of $\Sp_{2n}(\R)$. We fix the standard maximal compact subgroup
$$
 K=K^{(n)}=\{\mat{A}{B}{-B}{A}\in\GL_{2n}(\R):\:A\,^t\!A+B\,^tB=1,\;A\,^tB=B\,^t\!A\}.
$$
We have $K\cong U(n)$ via $\mat{A}{B}{-B}{A}\mapsto A+iB$. Let
$$
 j(g,Z)=\det(CZ+D)\qquad\text{for }g=\mat{A}{B}{C}{D}\in\Sp_{2n}(\R)
$$
and $Z$ in the Siegel upper half space $\mathbb{H}_n$. Then, for any integer $k$, the map
\begin{equation}\label{scalarKtypekeq}
 K\ni g\longmapsto j(g,I)^{-k},\qquad\text{where }I=iI_n,
\end{equation}
is a character of $K$.

Let $\mathfrak{h}$ be the compact Cartan subalgebra and $e_1,\ldots,e_n$ the linear forms on the complexification $\mathfrak{h}_\C$ defined in \cite{ASch}. A system of positive roots is given by $e_i\pm e_j$ for $1\leq i<j\leq n$ and $2e_j$ for $1\leq j\leq n$. The positive compact roots are the $e_i-e_j$ for $1\leq i<j\leq n$. The $K$-types are parametrized by the analytically integral elements $k_1e_1+\ldots+k_ne_n$, where the $k_i$ are integers with $k_1\geq\ldots\geq k_n$. We write $\rho_{\mathbf{k}}$, $\mathbf{k}=(k_1,\ldots,k_n)$, for the $K$-type with highest weight $k_1e_1+\ldots+k_ne_n$. If $k_1=\ldots=k_n=k$, then $\rho_{\mathbf{k}}$ is the $K_\infty$-type given in \eqref{scalarKtypekeq}; we simply write $\rho_k$ in this case.

The holomorphic discrete series representations of $\Sp_{2n}(\R)$ are parametrized by elements $\lambda=\ell_1e_1+\ldots\ell_ne_n$ with integers $\ell_1>\ldots>\ell_n>0$. The representation corresponding to the \emph{Harish-Chandra parameter} $\lambda$ contains the $K$-type $\rho_{\mathbf{k}}$, where $\mathbf{k}=\lambda+\sum_{j=1}^nje_j$, with multiplicity one; see Theorem 9.20 of \cite{knapp1986}. We denote this representation by $\dot\pi_\lambda$ or by $\pi_{\mathbf{k}}$; sometimes one or the other notation is more convenient. If $\mathbf{k}=(k,\ldots,k)$ with a positive integer $k>n$, then we also write $\pi_k$ for $\pi_{\mathbf{k}}$.

Let $G_{2n}(\R)^+$ be the index two subgroup of $G_{2n}(\R)$ consisting of elements with positive multiplier. We may extend a holomorphic discrete series representation $\dot\pi_\lambda$ of $\Sp_{2n}(\R)$ in a trivial way to $G_{2n}(\R)^+\cong\Sp_{2n}(\R)\times\R_{>0}$. This extension induces irreducibly to $G_{2n}(\R)$. We call the result a holomorphic discrete series representation of $G_{2n}(\R)$ and denote it by the same symbol $\dot\pi_\lambda$ (or $\pi_{\mathbf{k}}$). These are the archimedean components of the automorphic representations corresponding to vector-valued holomorphic Siegel modular forms of degree $n$.

\begin{lemma}\label{holdiscserembeddinglemma}
 Let $\lambda=\ell_1e_1+\ldots+\ell_ne_n$ with integers $\ell_1>\ldots>\ell_n>0$.
 \begin{enumerate}
  \item The holomorphic discrete series representation $\dot\pi_\lambda$ of $\Sp_{2n}(\R)$ embeds into
   \begin{equation}\label{holdiscserembeddinglemmaeq1}
    \sgn^{\ell_n+n}|\cdot|^{\ell_n}\times\ldots\times\sgn^{\ell_1+1}|\cdot|^{\ell_1}\rtimes1,
   \end{equation}
   and in no other principal series representation of $\Sp_{2n}(\R)$.
  \item The holomorphic discrete series representation $\dot\pi_\lambda$ of $G_{2n}(\R)$ embeds into
   \begin{equation}\label{holdiscserembeddinglemmaeq2}
    \sgn^{\ell_n+n}|\cdot|^{\ell_n}\times\ldots\times\sgn^{\ell_1+1}|\cdot|^{\ell_1}\rtimes\sgn^\varepsilon|\cdot|^{-\frac12(\ell_1+\ldots+\ell_n)},
   \end{equation}
   where $\varepsilon$ can be either $0$ or $1$, and in no other principal series representation of $G_{2n}(\R)$.
 \end{enumerate}
\end{lemma}
\begin{proof}
i) follows from the main result of \cite{Yamashita1989}. Part ii) can be deduced from i), observing that the holomorphic discrete series representations of $G_{2n}(\R)$ are invariant under twisting by the sign character.
\end{proof}

\begin{lemma}\label{degprincserhollemma}
 Let $k$ be a positive integer. Consider the degenerate principal series representation of $G_{2n}(\R)$ given by
 \begin{equation}\label{degprincserhollemmaeq1}
  J(s):=\sgn^k|\cdot|^{s-\frac{n+1}2}\rtimes\sgn^\varepsilon|\cdot|^{\frac{n(n+1)}4-\frac{ns}2},
 \end{equation}
 where $\varepsilon\in\{0,1\}$. Then $J(s)$ contains the holomorphic discrete series representation $\pi_k$ of $G_{2n}(\R)$ as a subrepresentation if and only if $s=k$.
\end{lemma}
\begin{proof}
By infinitesimal character considerations, we only need to prove the ``if'' part. Since $\pi_k$ is invariant under twisting by ${\rm sgn}$, we may assume that $\varepsilon=0$. Consider the Borel-induced representation
\begin{equation}\label{degprincserhollemmaeq2}
 J'(s):=\sgn^k|\cdot|^{s-n}\times\sgn^k|\cdot|^{s-n+1}\times\ldots\times\sgn^k|\cdot|^{s-1}\rtimes|\cdot|^{\frac{n(n+1)}4-\frac{ns}2}.
\end{equation}
By Lemma \ref{holdiscserembeddinglemma} ii), $\pi_k$ is a subrepresentation of $J'(k)$. Since $|\cdot|^{s-n}\times|\cdot|^{s-n+1}\times\ldots\times|\cdot|^{s-1}$ contains the trivial representation of $\GL_n(\R)$ twisted by $|\cdot|^{s-\frac{n+1}2}$, it follows that $J(s)\subset J'(s)$. Let $f_s$ be the function on $G_{2n}(\R)$ given by
\begin{equation}\label{degprincserhollemmaeq3}
 f_s(\mat{A}{*}{}{u\,^t\!A^{-1}}g,s)= \sgn^k(\det(A)) |u|^{-\frac{ns}2}|\det(A)|^s\,j(g,I)^{-k}
\end{equation}
for $A\in\GL_n(\R)$, $u\in\R^\times$ and $g\in K$. Then $f_s$ is a well-defined element of $J(s)$. Since $f_s$ is the unique up to multiples vector of weight $k$ in $J'(s)$, $\pi_k$ lies in the subspace $J(k)$ of $J'(k)$.
\end{proof}

Our method to calculate the local archimedean integrals \eqref{Rsflocaleq} will work for holomorphic discrete series representations $\dot\pi_\lambda$, where $\lambda=\ell_1e_1+\ldots+\ell_ne_n$ with $\ell_1>\ldots>\ell_n>0$ satisfies
\begin{equation}\label{lambdaconditioneq}
 \ell_j-\ell_{j-1}\text{ is odd for }2\leq j\leq n.
\end{equation}
Equivalently, we work with the holomorphic discrete series representations $\pi_{\mathbf{k}}$, where $\mathbf{k}=k_1e_1+\ldots+k_ne_n$ with $k_1\ge\ldots\ge k_n>n$ and all $k_i$ of the same parity; this last condition can be seen to be equivalent to \eqref{lambdaconditioneq}. An example for $\lambda$ satisfying \eqref{lambdaconditioneq} is $(k-1)e_1+\ldots+(k-n)e_n$, the Harish-Chandra parameter of $\pi_k$. The next lemma implies that whenever \eqref{lambdaconditioneq} is satisfied, then $\dot\pi_\lambda$ contains a convenient scalar $K$-type to work with.
\begin{lemma}\label{scalarKtypeslemma}
 Assume that $\lambda=\ell_1e_1+\ldots+\ell_ne_n$ with $\ell_1>\ldots>\ell_n>0$ satisfying \eqref{lambdaconditioneq}. Let $m$ be a non-negative, even integer. Then the holomorphic discrete series representation $\dot\pi_\lambda$ contains the $K$-type
 \begin{equation}\label{scalarKtypeslemmaeq1}
  \rho_{\mathbf{m}},\qquad \mathbf{m}=(\ell_1+1+m)(e_1+\ldots+e_n),
 \end{equation}
 with multiplicity one.
\end{lemma}
\begin{proof}
By Theorem 8.1 of \cite{knapp1986} we need only show that $\rho_{\mathbf{m}}$ occurs in $\dot\pi_\lambda$. We will use induction on $n$. The result is obvious for $n=1$. Assume that $n>1$, and that the assertion has already been proven for $n-1$.

Using standard notations as in \cite{ASch}, we have $\mathfrak{g}_\C=\mathfrak{p}_\C^-\oplus\mathfrak{k}_\C\oplus\mathfrak{p}_\C^+$. The universal enveloping algebra of $\mathfrak{p}_\C^+$ is isomorphic to the symmetric algebra $S(\mathfrak{p}_\C^+)$. We have $\dot\pi_\lambda\cong S(\mathfrak{p}_\C^+)\otimes\rho_\lambda$ as $K$-modules. Let $I$ be the subalgebra of $S(\mathfrak{p}_\C^+)$ spanned by the highest weight vectors of its $K$-types. By Theorem A of \cite{Johnson1980}, there exists in $I$ an element $D_+$ of weight $2(e_1+\ldots+e_n)$. By the main result of \cite{HoweKraft1998}, the space of $K$-highest weight vectors of $\dot\pi_\lambda$ is acted upon freely by $I$. It follows that we need to prove our result only for $m=0$.

We will use the Blattner formula proven in \cite{HechtSchmid1975}. It says that the multiplicity with which $\rho_{\mathbf{m}}$ occurs in $\dot\pi_\lambda$ is given by
\begin{equation}\label{scalarKtypeslemmaeq2}
 {\rm mult}_\lambda(\mathbf{m})=\sum_{w\in W_K}\varepsilon(w)Q(w(\mathbf{m}+\rho_c)-\lambda-\rho_n).
\end{equation}
Here, $W_K$ is the compact Weyl group, which in our case is isomorphic to the symmetric group $S_n$, and $\varepsilon$ is the sign character on $W_K$; the symbols $\rho_c$ and $\rho_n$ denote the half sums of the positive compact and non-compact roots, respectively; and $Q(\mu)$ is the number of ways to write $\mu$ as a sum of positive non-compact roots. In our case
\begin{equation}\label{scalarKtypeslemmaeq3}
 \rho_c=\frac12\sum_{j=1}^n(n+1-2j)e_j,\qquad\rho_n=\frac{n+1}2\sum_{j=1}^ne_j.
\end{equation}
Hence
\begin{align}\label{scalarKtypeslemmaeq4}
 {\rm mult}_\lambda(\mathbf{m})&=\sum_{\sigma\in S_n}\varepsilon(\sigma)Q\Big(\sigma\Big((\ell_1+1+m)\sum_{j=1}^ne_j+\frac12\sum_{j=1}^n(n+1-2j)e_j\Big)-\lambda-\frac{n+1}2\sum_{j=1}^ne_j\Big)\nonumber\\
  &=\sum_{\sigma\in S_n}\varepsilon(\sigma)Q\Big(\sum_{j=1}^n(\ell_1+1+m-\ell_j)e_j-\sigma\Big(\sum_{j=1}^nje_j\Big)\Big).
\end{align}
Now assume that $m=0$. Then
\begin{equation}\label{scalarKtypeslemmaeq5}
 {\rm mult}_\lambda(\mathbf{m})=\sum_{\sigma\in S_n}\varepsilon(\sigma)Q\Big(\sum_{j=1}^n(\ell_1+1-\ell_j)e_j-\sigma\Big(\sum_{j=1}^nje_j\Big)\Big).
\end{equation}
If $\sigma(1)\neq1$, then the coefficient of $e_1$ is negative, implying that $Q(\ldots)=0$. Hence
\begin{align}\label{scalarKtypeslemmaeq6}
 {\rm mult}_\lambda(\mathbf{m})&=\sum_{\substack{\sigma\in S_n\\\sigma(1)=1}}\varepsilon(\sigma)Q\Big(\sum_{j=2}^n(\ell_1+1-\ell_j)e_j-\sigma\Big(\sum_{j=2}^nje_j\Big)\Big)\nonumber\\
 &=\sum_{\substack{\sigma\in S_n\\\sigma(1)=1}}\varepsilon(\sigma)Q\Big(\sum_{j=1}^{n-1}(\ell_1-\ell_{j+1})e_{j+1}-\sigma\Big(\sum_{j=1}^{n-1}je_{j+1}\Big)\Big).
\end{align}
If we set $e'_j=e_{j+1}$ and $m'=\ell_1-\ell_2-1$, then this can be written as
\begin{equation}\label{scalarKtypeslemmaeq7}
 {\rm mult}_\lambda(\mathbf{m})=\sum_{\sigma\in S_{n-1}}\varepsilon(\sigma)Q\Big(\sum_{j=1}^{n-1}(\ell_2+1+m'-\ell_{j+1})e'_j-\sigma\Big(\sum_{j=1}^{n-1}je'_j\Big)\Big).
\end{equation}
We see that this is the formula \eqref{scalarKtypeslemmaeq4}, with $n-1$ instead of $n$ and $m'$ instead of $m$; note that $m'$ is even and non-negative by our hypotheses. There are two different $Q$-functions involved, for $n$ and for $n-1$, but since the argument of $Q$ in \eqref{scalarKtypeslemmaeq6} has no $e_1$, we may think of it as the $Q$-function for $n-1$. Therefore \eqref{scalarKtypeslemmaeq7} represents the multiplicity of the $K^{(n-1)}$-type $(\ell_2+1+m')(e'_1+\ldots+e'_{n-1})$ in the holomorphic discrete series representation $\dot\pi_{\lambda'}$ of $\Sp_{2(n-1)}(\R)$, where $\lambda'=\ell_2e'_1+\ldots+\ell_ne'_{n-1}$. By induction hypothesis, this multiplicity is $1$, completing our proof.
\end{proof}
\subsection{Calculating the integral}\label{reallocalsec}
In the remainder of this section we fix a real place $v$ and calculate the local archimedean integral \eqref{Rsflocaleq} for a certain choice of vectors $f$ and $w$. To ease notation, we omit the subscript $v$. We assume that the underlying representation $\pi$ of $G_{2n}(\R)$ is a holomorphic discrete series representation $\dot\pi_\lambda$, where $\lambda=\ell_1e_1+\ldots+\ell_ne_n$ with $\ell_1>\ldots>\ell_n>0$ satisfies \eqref{lambdaconditioneq}. Set $k=\ell_1+1$. By Lemma \ref{scalarKtypeslemma}, the $K$-type $\rho_k$ appears in $\pi$ with multiplicity $1$. Let $w_\lambda$ be a vector spanning this one-dimensional $K$-type. We choose $w=w_\lambda$ as our vector in the zeta integral $Z(s,f,w)$.

To explain our choice of $f$, let $J(s)$ be the degenerate principal series representation of $G_{4n}(\R)$ defined in Lemma \ref{degprincserhollemma} (hence, we replace $n$ by $2n$ in \eqref{degprincserhollemmaeq1}). We see from \eqref{Ichisaltnoteq} that $I(\sgn^k,s)$ equals $J((2n+1)(s+\frac12))$ for appropriate $\varepsilon\in\{0,1\}$. Let $f_k(\cdot,s)$ be the vector spanning the $K$-type $K^{(2n)}\ni g\longmapsto j(g,I)^{-k}$ and normalized by $f_k(1,s)=1$. Explicitly,
\begin{equation}\label{degprincserhollemmaeq3c}
 f_k(\mat{A}{*}{}{u\,^t\!A^{-1}}g,s)= \sgn^k(\det(A))\,|u^{-n}\det(A)|^{(2n+1)(s+\frac12)}\,j(g,I)^{-k}
\end{equation}
for $A\in\GL_{2n}(\R)$, $u\in\R^\times$ and $g\in K^{(2n)}$. Then $f=f_k$ is the section which we will put in our local archimedean integral $Z(s,f,w)$.

Thus consider $Z(s,f_k,w_\lambda)$. By Proposition \ref{interopprop} i), this integral is a vector in $\pi$. The observation \eqref{Qninveq2}, together with the transformation properties of $f_k$, imply that for $g \in K$
\begin{equation}\label{Zfkwkeq1}
 \pi(g)Z(s,f_k,w_\lambda)=j(g,I)^{-k}Z(s,f_k,w_\lambda).
\end{equation}
Since the $K$-type $\rho_k$ occurs only once in $\pi$, it follows that
\begin{equation}\label{Zfkwkeq2}
 Z(s,f_k,w_\lambda)=B_\lambda(s)w_\lambda
\end{equation}
for a constant $B_\lambda(s)$ depending on $s$ and the Harish-Chandra parameter $\lambda$. The rest of this section is devoted to calculating $B_\lambda(s)$ explicitly.
\subsubsection*{The scalar minimal $K$-type case}
We first consider the case $\lambda=(k-1)e_1+\ldots+(k-n)e_n$ with $k>n$. Then $\dot\pi_\lambda=\pi_k$, the holomorphic discrete series representation of $G_{2n}(\R)$ with minimal $K$-type $\rho_k$. Let $w_k$ be a vector spanning this minimal $K$-type. Let $\langle\,,\,\rangle$ be an appropriately normalized invariant hermitian inner product on the space of $\pi_k$ such that $\langle w_k,w_k\rangle=1$. As proved in the appendix of \cite{KnightlyLi2016}, we have the following simple formula for the corresponding matrix coefficient,
\begin{equation}\label{matrixcoeffeq}
 \langle\pi_k(h)w_k,w_k\rangle=\begin{cases}
       \displaystyle\frac{\mu_n(h)^{nk/2}\,2^{nk}}{\det{(A+D+i(C-B))^k}}&\text{for }\mu(h)>0,\\[2ex]
       0&\text{for }\mu(h)<0.
                           \end{cases}
\end{equation}
Here, $h=\mat{A}{B}{C}{D}\in G_{2n}(\R)$. We will need the following result.
\begin{lemma}\label{Tintegrallemma}
 For a complex number $z$, let
 \begin{equation}\label{Tintegrallemmaeq1}
  \gamma_n(z)=\int\limits_T\bigg(\prod_{1\leq i<j\leq n}(t_i^2-t_j^2)\bigg)\Big(\prod_{j=1}^nt_j\Big)^{-z-n}\,d\mathbf{t},
 \end{equation}
 where $d\mathbf{t}$ is the Lebesgue measure and
 \begin{equation}\label{Blambdaformulaeq4}
  T=\{(t_1,\ldots,t_n)\in\R^n\::\:t_1>\ldots>t_n>1\}.
 \end{equation}
 Then, for real part of $z$ large enough,
 \begin{equation}\label{gammanlemmaeq2}
  \gamma_n(z)=\prod_{m=1}^{n}(m-1)!\prod_{j=1}^m\frac{1}{z-m-1+2j}.
 \end{equation}
\end{lemma}
\begin{proof}
After some straightforward variable transformations, our integral reduces to the Selberg integral; see \cite{Selberg1944} or \cite{ForresterWarnaar2008}.
\end{proof}

\begin{proposition}\label{scalarminKtypeprop}
 Assume that $k>n$ and $\lambda=(k-1)e_1+\ldots+(k-n)e_n$, so that $\dot\pi_\lambda=\pi_k$. Then
 \begin{equation}\label{scalarminKtypepropeq1}
  B_\lambda(s)=\frac{\pi^{n(n+1)/2}}{\prod_{m=1}^n(m-1)!}\,i^{nk}\,2^{-n(2n+1)s+3n/2}\,\gamma_n\Big((2n+1)s-\frac12+k\Big),
 \end{equation}
 where  $\gamma_n$ is the rational function \eqref{gammanlemmaeq2}.
\end{proposition}
\begin{proof}
Taking the inner product with $w_k$ on both sides of \eqref{Zfkwkeq2}, we obtain
\begin{equation}\label{Blambdaformulaeq1}
 B_\lambda(s)=\int\limits_{\Sp_{2n}(\R)} f_k(Q_n \cdot (h, 1), s)\langle\pi_k(h)w_k,w_k\rangle\,dh.
\end{equation}
Since the integrand is left and right $K$-invariant, we may apply the integration formula \eqref{KAKintegrationeq}. Thus
\begin{equation}\label{Blambdaformulaeq2}
 B_\lambda(s)=\alpha_n\int\limits_{\mathfrak{a}^+}\bigg(\prod_{\lambda\in\Sigma^+}\sinh(\lambda(H))\bigg)f_k(Q_n \cdot (\exp(H), 1), s)\langle\pi_k(\exp(H))w_k,w_k\rangle\,dH.
\end{equation}
This time we work with the positive system $\Sigma^+=\{e_i-e_j:1\leq i<j\leq n\}\,\cup\,\{e_i+e_j:1\leq i\leq j\leq n\}$, so that
\begin{equation}\label{posweylchambereq2}
 \mathfrak{a}^+=\{{\rm diag}(a_1,\ldots,a_n,-a_1,\ldots,-a_n):a_1>\ldots>a_n>0\}.
\end{equation}
The function $f_k$ in \eqref{Blambdaformulaeq2} can be evaluated as follows,
\begin{align*}
 f_k(Q_n \cdot (\exp(H), 1), s)&=f_k(\left[\begin{smallmatrix}I_n\\&I_n\\&I_n&I_n\\I_n&&&I_n\end{smallmatrix}\right]\left[\begin{smallmatrix}I_n\\&&&I_n\\&&I_n\\&-I_n\end{smallmatrix}\right] \cdot (\exp(H), 1), s)\\
 &=i^{nk}f_k(\left[\begin{smallmatrix}I_n\\&I_n\\&I_n&I_n\\I_n&&&I_n\end{smallmatrix}\right]\left[\begin{smallmatrix}A\\&I_n\\&&A^{-1}\\&&&I_n\end{smallmatrix}\right], s)\qquad( A=\left[\begin{smallmatrix}e^{a_1}\\&\ddots\\&&e^{a_n}\end{smallmatrix}\right])\\
  &=i^{nk}e^{(a_1+\ldots+a_n)(2n+1)(s+\frac12)}f_k(\left[\begin{smallmatrix}I_n\\&I_n\\&A&I_n\\A&&&I_n\end{smallmatrix}\right], s).
\end{align*}
Now use
\begin{equation}\label{SL2NAKeq4b}
 \mat{1}{}{e^{a_j}}{1}=\mat{1}{x}{}{1}\mat{y^{1/2}}{}{}{y^{-1/2}}r(\theta),\qquad r(\theta)=\mat{\cos(\theta)}{\sin(\theta)}{-\sin(\theta)}{\cos(\theta)},
\end{equation}
with
\begin{equation}\label{SL2NAKeq5b}
 x=\frac{e^{a_j}}{1+e^{2a_j}},\qquad y=\frac1{1+e^{2a_j}},\qquad e^{i\theta}=\frac{1-ie^{a_j}}{(1+e^{2a_j})^{1/2}}.
\end{equation}
It follows that
\begin{equation}\label{fkevaleq}
 f_k(Q_n \cdot (\exp(H), 1), s)=i^{nk}\prod_{j=1}^n(e^{a_j}+e^{-a_j})^{-(2n+1)(s+\frac12)}.
\end{equation}
Substituting this and \eqref{matrixcoeffeq} into \eqref{Blambdaformulaeq2}, we get after some simplification

\begin{align*}
 \alpha_n^{-1}B_\lambda(s)
 & =i^{nk}\,2^{-n(2n+1)(s+\frac12)+n}\\ & \times \int\limits_{\mathfrak{a}^+}\prod_{1\leq i<j\leq n}(\cosh(a_i)^2-\cosh(a_j)^2)\prod_{j=1}^n
  \cosh(a_j)^{-(2n+1)s+\frac12-n-k}
  \prod_{j=1}^n\sinh(a_j)\,dH.
\end{align*}
Now introduce the new variables $t_j=\cosh(a_j)$. The domain $\mathfrak{a}^+$ turns into the domain $T$ defined in \eqref{Blambdaformulaeq4}. We get
\begin{align}\label{Blambdaformulaeq5}
 B_\lambda(s)&=\alpha_ni^{nk}\,2^{-n(2n+1)(s+\frac12)+n}\int\limits_T\bigg(\prod_{1\leq i<j\leq n}(t_i^2-t_j^2)\bigg)\Big(\prod_{j=1}^nt_j\Big)^{-(2n+1)s+\frac12-n-k}\,d\mathbf{t}.
\end{align}
Thus our assertion follows from Lemma \ref{Tintegrallemma} and the value of $\alpha_n$ given in \eqref{alphanpropeq1}.
\end{proof}
\subsubsection*{Calculation of $B_\lambda(s)$ for general Harish-Chandra parameter}
Now consider a general $\lambda=\ell_1e_1+\ldots+\ell_ne_n$ for which the condition \eqref{lambdaconditioneq} is satisfied. By \eqref{holdiscserembeddinglemmaeq2}, we may embed $\dot\pi_\lambda$ into
\begin{equation}\label{holdiscserembeddinglemmaeq2b}
    \sgn^k|\cdot|^{\ell_n}\times\ldots\times\sgn^k|\cdot|^{\ell_1}\rtimes|\cdot|^{-\frac12(\ell_1+\ldots+\ell_n)}.
\end{equation}
In this induced model, the weight $k$ vector $w_\lambda$ in $\dot\pi_\lambda$ has the formula
\begin{align}\label{weight-k-vector}
 w_\lambda(\left[\begin{smallmatrix}
  a_1&\cdots&*&*&\cdots&*\\&\ddots&\vdots&\vdots&&\vdots\\&&a_n&*&\cdots&*\\&&&a_0 a_1^{-1}&&\\&&&\vdots&\ddots&\\&&&*&\cdots&a_0a_n^{-1}
 \end{smallmatrix}\right]g)&=\sgn(a_1\cdot\ldots\cdot a_n)^k\,|a_0|^{-\frac12(\ell_1+\ldots+\ell_n)-\frac{n(n+1)}4}\nonumber\\
 &\hspace{10ex}\bigg(\prod_{j=1}^n|a_j|^{\ell_{n+1-j}+n+1-j}\bigg)j(g,I)^{-k}
\end{align}
for $a_1,\ldots,a_n\in \R^\times$ and $g \in K^{(n)}$. Evaluating \eqref{Zfkwkeq2} at $1$, we get
\begin{equation}\label{Blambdaintegraleq}
 B_\lambda(s)=\int\limits_{\Sp_{2n}(\R)}f_k(Q_n\cdot(h,1),s)w_\lambda(h)\,dh.
\end{equation}
Recall the beta function
\begin{equation}\label{betafcteq1}
 B(x,y)=\frac{\Gamma(x)\Gamma(y)}{\Gamma(x+y)}.
\end{equation}
One possible integral representation for the beta function is
\begin{equation}\label{betafcteq2}
 B(x,y)=\int\limits_0^\infty\frac{a^{x-1}}{(a+1)^{x+y}}\,da=2\int\limits_0^\infty\frac{a^{x-y}}{(a+a^{-1})^{x+y}}\,d^\times a\qquad\text{for }{\rm Re}(x),{\rm Re}(y)>0.
\end{equation}
For $s\in\C$ and $m\in\Z$, let
\begin{equation}\label{betamsdef}
 \beta(m,s)=B\Big(\Big(n+\frac12\Big)s+\frac14+\frac{m}2,\Big(n+\frac12\Big)s+\frac14-\frac{m}2\Big).
\end{equation}

\begin{lemma}\label{Blambdanlemma1}
 We have  \begin{equation}\label{Blambdanlemma1eq1}
  B_\lambda(s)=i^{nk}\bigg(\prod_{j=1}^n\beta(\ell_j,s)\bigg)C_k(s),
 \end{equation}
  where
 \begin{equation}\label{Blambdanlemma1eq2}
  C_k(s)=\int\limits_{\tilde N_1}\int\limits_{\tilde N_2}f_k(\left[\begin{smallmatrix}I_n\\&I_n\\&^tV&I_n\\V&X&&I_n\end{smallmatrix}\right],s)\,dX\,dV
 \end{equation}
 depends only on $k=\ell_1+1$. Here, $\tilde N_1$ is the space of upper triangular nilpotent matrices of size $n\times n$, and $\tilde N_2$ is the space of symmetric $n\times n$ matrices.
\end{lemma}
\begin{proof}
In this proof we write $f_k(h)$ instead of $f_k(h,s)$ for simplicity. We will use the Iwasawa decomposition to calculate the integral \eqref{Blambdaintegraleq}. By Proposition \ref{Iwasawameasureprop}, the relevant integration formula for right $K^{(n)}$-invariant functions $\varphi$ is
\begin{equation}\label{Sp2niwasawainteq2}
 \int\limits_{\Sp_{2n}(\R)}\varphi(h)\,dh=2^n\int\limits_A\int\limits_N\varphi(an)\,dn\,da,
\end{equation}
where $N$ is the unipotent radical of the Borel subgroup, $dn$ is the Lebesgue measure, $A=\{{\rm diag}(a_1,\ldots,a_n,a_1^{-1},\ldots,a_n^{-1})\,:\,a_1,\ldots,a_n>0\}$, and $da=\frac{da_1}{a_1}\ldots\frac{da_n}{a_n}$. Thus, using \eqref{weight-k-vector},
\begin{align}\label{Blambdanlemma1eq4}
 B_\lambda(s)&=2^n\int\limits_A\int\limits_Nf_k(Q_n\cdot(an,1))w_\lambda(an)\,dn\,da\nonumber\\
 &=2^n\int\limits_A\int\limits_N\bigg(\prod_{j=1}^na_j^{\ell_{n+1-j}+n+1-j}\bigg)f_k(Q_n\cdot(an,1))\,dn\,da.
\end{align}
We have $N=N_1N_2$, where
\begin{equation}\label{Blambdanlemma1eq5}
 N_1=\{\mat{U}{}{}{^tU^{-1}}:U\text{ upper triangular unipotent}\},\qquad N_2=\{\mat{I_n}{X}{}{I_n}:X\text{ symmetric}\}.
\end{equation}
Write further $a=\mat{D}{}{}{D^{-1}}$ with $D={\rm diag}(a_1,\ldots,a_n)$. Then
\begin{align}\label{Blambdanlemma1eq6}
 f_k(Q_n\cdot(an,1))&=f_k(\left[\begin{smallmatrix}I_n\\&I_n\\&I_n&I_n\\I_n&&&I_n\end{smallmatrix}\right]\left[\begin{smallmatrix}I_n\\&&&I_n\\&&I_n\\&-I_n\end{smallmatrix}\right]\left[\begin{smallmatrix}D\\&I_n\\&&D^{-1}\\&&&I_n\end{smallmatrix}\right]\left[\begin{smallmatrix}U\\&I_n\\&&^tU^{-1}\\&&&I_n\end{smallmatrix}\right]\left[\begin{smallmatrix}I_n&&-X\\&I_n\\&&I_n\\&&&I_n\end{smallmatrix}\right])\nonumber\\
 &=i^{nk}|a_1\cdot\ldots\cdot a_n|^{(2n+1)(s+\frac12)}f_k(\left[\begin{smallmatrix}I_n\\&I_n\\&D&I_n\\D&&&I_n\end{smallmatrix}\right]\left[\begin{smallmatrix}U\\&I_n\\&&^tU^{-1}\\&&&I_n\end{smallmatrix}\right]\left[\begin{smallmatrix}I_n&&-X\\&I_n\\&&I_n\\&&&I_n\end{smallmatrix}\right]).
\end{align}
We will use the identity
\begin{equation}\label{Blambdanlemma1eq7}
 \mat{1}{}{a_j}{1}=\mat{1}{x_j}{}{1}\mat{y_j^{1/2}}{}{}{y_j^{-1/2}}r(\theta_i)
\end{equation}
with
\begin{equation}\label{Blambdanlemma1eq8}
 x_j=\frac{a_j}{1+a_j^2},\qquad y_j=\frac1{1+a_j^2},\qquad e^{i\theta_j}=\frac{1-ia_j}{(1+a_j^2)^{1/2}}
\end{equation}
Let
\begin{equation}\label{Blambdanlemma1eq9}
 Y=\left[\begin{smallmatrix}y_1^{1/2}\\&\ddots\\&&y_n^{1/2}\end{smallmatrix}\right],\qquad
 C=\left[\begin{smallmatrix}\cos(\theta_1)\\&\ddots\\&&\cos(\theta_n)\end{smallmatrix}\right],\qquad
 S=\left[\begin{smallmatrix}\sin(\theta_1)\\&\ddots\\&&\sin(\theta_n)\end{smallmatrix}\right].
\end{equation}
Then
\begin{equation}\label{Blambdanlemma1eq10}
 \left[\begin{smallmatrix}I_n\\&I_n\\&D&I_n\\D&&&I_n\end{smallmatrix}\right]=\left[\begin{smallmatrix}Y&&&*\\&Y&*\\&&Y^{-1}\\&&&Y^{-1}\end{smallmatrix}\right]\left[\begin{smallmatrix}C&&&S\\&C&S\\&-S&C\\-S&&&C\end{smallmatrix}\right],
\end{equation}
and thus
\begin{align}\label{Blambdanlemma1eq11}
 f_k(Q_n\cdot(an,1))&=i^{nk}\bigg(\prod_{j=1}^n\frac1{a_j+a_j^{-1}}\bigg)^{(2n+1)(s+\frac12)}\!f_k(\left[\begin{smallmatrix}C&&&S\\&C&S\\&-S&C\\-S&&&C\end{smallmatrix}\right]\!\!\left[\begin{smallmatrix}U\\&I_n\\&&^tU^{-1}\\&&&I_n\end{smallmatrix}\right]\!\!\left[\begin{smallmatrix}I_n&&-X\\&I_n\\&&I_n\\&&&I_n\end{smallmatrix}\right]).
\end{align}
Write $U=I_n+V$, so that $V$ is upper triangular nilpotent. A calculation confirms that
\begin{equation}\label{Blambdanlemma1eq12}
 \left[\begin{smallmatrix}C&&&S\\&C&S\\&-S&C\\-S&&&C\end{smallmatrix}\right]\left[\begin{smallmatrix}U\\&I_n\\&&^tU^{-1}\\&&&I_n\end{smallmatrix}\right]=p\left[\begin{smallmatrix}I_n\\&I_n\\&^tZ&I_n\\Z&&&I_n\end{smallmatrix}\right]\left[\begin{smallmatrix}C&&&S\\&C&S\\&-S&C\\-S&&&C\end{smallmatrix}\right],
\end{equation}
where $Z=-(1+SVS)^{-1}SVC$ and $p=\mat{B}{*}{}{^tB^{-1}}$ with $\det(B)=1$. Hence
\begin{align}\label{Blambdanlemma1eq13}
 f_k(Q_n\cdot(an,1))&=i^{nk}\bigg(\prod_{j=1}^n\frac1{a_j+a_j^{-1}}\bigg)^{(2n+1)(s+\frac12)}\!f_k(\left[\begin{smallmatrix}I_n\\&I_n\\&^tZ&I_n\\Z&&&I_n\end{smallmatrix}\right]\!\left[\begin{smallmatrix}C&&&S\\&C&S\\&-S&C\\-S&&&C\end{smallmatrix}\right]\!\left[\begin{smallmatrix}I_n&&-X\\&I_n\\&&I_n\\&&&I_n\end{smallmatrix}\right]).
\end{align}
Let $\tilde N_1$ be the Euclidean space of upper triangular nilpotent real matrices of size $n\times n$. Then it is an exercise to verify that
\begin{equation}\label{Blambdanlemma1eq14}
 \varphi\longmapsto\int\limits_{\tilde N_1}\varphi(I_n+V)\,dV,
\end{equation}
where $dV$ is the Lebesgue measure, defines a Haar measure on the group of upper triangular unipotent real matrices. (Use the fact that $V\mapsto UV$ defines an automorphism of $\tilde N$ of determinant $1$, for every upper triangular unipotent $U$.)

Therefore, as we integrate \eqref{Blambdanlemma1eq13} over $N_1$, we may treat $V$ as a Euclidean variable. We then have to consider the Jacobian of the change of variables $V\mapsto Z$. It is not difficult to show that this Jacobian is $\prod_{j=1}^n\sin(\theta_j)^{j-n}\cos(\theta_j)^{1-j}$. Substituting from \eqref{Blambdanlemma1eq8}, we find
\begin{equation}\label{Blambdanlemma1eq16}
 \bigg|\prod_{j=1}^n\sin(\theta_j)^{j-n}\cos(\theta_j)^{1-j}\bigg|=\bigg(\prod_{j=1}^na_j\bigg)^{-\frac{n+1}2}\bigg(\prod_{j=1}^na_j^j\bigg)\bigg(\prod_{j=1}^n\frac1{a_j+a_j^{-1}}\bigg)^{\frac{1-n}2}.
\end{equation}
Using the above and some more matrix identities, we get
\begin{align}\label{Blambdanlemma1eq17}
 \int\limits_{N_1}f_k(Q_n\cdot(an_1n_2,1))\,dn_1&=i^{nk}\bigg(\prod_{j=1}^n\frac1{a_j+a_j^{-1}}\bigg)^{(2n+1)s+\frac12}\bigg(\prod_{j=1}^na_j\bigg)^{-n-1}\bigg(\prod_{j=1}^na_j^j\bigg)\nonumber\\
 &\hspace{2ex}\int\limits_{\tilde N_1}\int\limits_{\tilde N_2}f_k(\left[\begin{smallmatrix}I_n\\&I_n\\&^tV&I_n\\V&&&I_n\end{smallmatrix}\right]\left[\begin{smallmatrix}I_n\\&I_n\\&&I_n\\&X&&I_n\end{smallmatrix}\right])\,dX\,dV.
\end{align}
This last integral is the $C_k(s)$ defined in \eqref{Blambdanlemma1eq2}. Going back to \eqref{Blambdanlemma1eq4} and using \eqref{betafcteq2}, we  have
\begin{align}\label{Blambdanlemma1eq22}
 B_\lambda(s)&=2^n\int\limits_A\int\limits_{N_1}\int\limits_{N_2}\bigg(\prod_{j=1}^na_j^{\ell_{n+1-j}+n+1-j}\bigg)f_k(Q_n\cdot(an_1n_2,1))\,dn_2\,dn_1\,da\nonumber\\
 &=2^ni^{nk}\int\limits_A\bigg(\prod_{j=1}^na_j^{\ell_{n+1-j}}\bigg)\bigg(\prod_{j=1}^n\frac1{a_j+a_j^{-1}}\bigg)^{(2n+1)s+\frac12}\,da\cdot C_k(s)\nonumber\\
   &=i^{nk}\prod_{j=1}^n\bigg(B\Big(\Big(n+\frac12\Big)s+\frac14+\frac{\ell_j}2,\Big(n+\frac12\Big)s+\frac14-\frac{\ell_j}2\Big)\bigg)\cdot C_k(s).
\end{align}
This concludes the proof.
\end{proof}

\begin{lemma}\label{Blambdanlemma2}
 The function $C_k(s)$ defined in \eqref{Blambdanlemma1eq2} is given by
 \begin{equation}\label{Blambdanlemma2eq1}
  C_k(s)=\frac{\pi^{n(n+1)/2}}{\prod_{m=1}^{n}(m-1)!}2^{-n(2n+1)s+3n/2}\,\frac{\gamma_n((2n+1)s-\frac12+k)}{\prod_{j=1}^n\beta(k-j,s)},
 \end{equation}
  where $\gamma_n$ is the rational function from Lemma \ref{Tintegrallemma}, and $\beta(m,s)$ is defined in \eqref{betamsdef}.
\end{lemma}
\begin{proof}
Consider $\lambda=(k-1)e_1+\ldots+(k-n)e_n$. Then
\begin{equation}\label{Blambdanlemma2eq2}
  B_\lambda(s)=\frac{\pi^{n(n+1)/2}}{\prod_{m=1}^n(m-1)!}\,i^{nk}\,2^{-n(2n+1)s+3n/2}\,\gamma_n\Big((2n+1)s-\frac12+k\Big),
\end{equation}
by Proposition \ref{scalarminKtypeprop}. On the other hand,
\begin{equation}\label{Blambdanlemma2eq3}
  B_\lambda(s)=i^{nk}\,\bigg(\prod_{j=1}^n\beta(k-j,s)\bigg)C_k(s)
\end{equation}
by Lemma \ref{Blambdanlemma1}. The assertion follows by comparing the two expressions.
\end{proof}

\begin{proposition}\label{archzetaprop}
 Let $\lambda=\ell_1e_1+\ldots+\ell_ne_n$, $\ell_1>\ldots>\ell_n>0$, be a Harish-Chandra parameter for which the condition \eqref{lambdaconditioneq} is satisfied. Denote $k_j=\ell_j+j$ and $k=k_1$. Let $\pi=\dot\pi_\lambda$ be the corresponding holomorphic discrete series representation of $G_{2n}(\R)$  and let $w_\lambda$ be a vector in $\pi$ spanning the $K$-type $\rho_k$. Let $f_k$ be the element of $I(\sgn^k,s)$ given by \eqref{degprincserhollemmaeq3c}. Then
 $$
  Z(s,f_k,w_\lambda)=i^{nk}\,\pi^{n(n+1)/2} A_{\mathbf{k}}((2n+1)s- 1/2) \, w_\lambda
 $$
 with  the function $A_{\mathbf{k}}(z)$ is defined as
 \begin{equation}\label{archzetapropeq2}
  A_{\mathbf{k}}(z)=2^{-n(z-1)}\bigg(\prod_{j=1}^n
  \prod_{i=1}^j\frac{1}{z+k-1-j+2i}\bigg)\bigg(\prod_{j=1}^n
  \prod_{i=0}^{\frac{k-k_j}2-1}\frac{z-(k-1-j-2i)}{z+(k-1-j-2i)}\bigg).
 \end{equation}
\end{proposition}
\begin{proof}
By Lemma \ref{Blambdanlemma1} and Lemma \ref{Blambdanlemma2}, $Z(s,f_k,w_\lambda) = B_\lambda(s) w_\lambda$ with
\begin{equation}\label{archzetapropeq3}
  B_\lambda(s)=i^{nk}\,\frac{\pi^{n(n+1)/2}}{\prod_{m=1}^{n}(m-1)!}2^{-n(2n+1)s+3n/2}\,\gamma_n\Big((2n+1)s-\frac12+k\Big)\prod_{j=1}^n\frac{\beta(\ell_j,s)}{\beta(k-j,s)}.
\end{equation}
Inductively one confirms the identity
\begin{equation}\label{archzetapropeq4}
 B(x+m,y-m)=B(x,y)\prod_{i=0}^{m-1}\frac{x+i}{y-i-1}
\end{equation}
for any integer $m\geq0$. Use the abbreviation $t=(n+\frac12)s+\frac14$. Applying the above formula with $m=\frac{k-k_j}2$, which by \eqref{lambdaconditioneq} is a non-negative integer, and replacing $i$ by $m-1-i$, we get
\begin{equation}
 \frac{\beta(\ell_j,s)}{\beta(k-j,s)}
=\prod_{i=0}^{m-1}\frac{t-\frac{\ell_j}2-m+i}{t+\frac{\ell_j}2+m-1-i}.
\end{equation}
The result follows by using formula \eqref{gammanlemmaeq2} for $\gamma_n$.\end{proof}
\begin{remark}\label{rem:Akrational}
Using \eqref{archzetapropeq2}, one can check that $A_\mathbf{k}(t)$ is a non-zero rational number for any integer $t$ satisfying $0 \le t \le k_n-n$.
\end{remark}
\section{The global integral representation}\label{global-int-section}
\subsection{The main result}\label{global-int-section-main}
Consider the global field $F = \Q$ and its ring of adeles $\A = \A_\Q$. All the results are easily generalizable to a totally real number field. Let $\pi\cong\otimes \pi_p$ be a cuspidal automorphic representation of $G_{2n}(\A)$. We assume that $\pi_\infty$ is a holomorphic discrete series representation $\pi_{\mathbf{k}}$ with ${\mathbf{k}}=k_1e_1+\ldots+k_ne_n$, where $k_1\ge\ldots\ge k_n>n$ and all $k_i$ have the same parity. (From now on it is more convenient to work with the minimal $K$-type $\mathbf{k}$ rather than the Harish-Chandra parameter $\lambda$.) We set $k=k_1$. Let $\chi=\otimes\chi_p$ be a character of $\Q^\times \bs \A^\times$ such that $\chi_\infty = \sgn^k$. Let $N=\prod_{p|N} p^{m_p}$ be an integer such that \begin{itemize}
\item For each finite prime $p \nmid N$ both $\pi_p$ and $\chi_p$ are unramified.
 \item For a  prime $p |N$, we have $\chi_p |_{(1+p^{m_p}\Z_p) \cap \Z_p^\times} = 1$ and $\pi_p$ has a vector $\phi_p$ that is right invariant under the principal congruence subgroup $\Gamma_{2n}(p^{m_p})$ of $\Sp_{2n}(\Z_p)$.
\end{itemize}

Let $\phi$ be a cusp form in the space of $\pi$ corresponding to a pure tensor $\otimes \phi_p$, where the local vectors are chosen as follows. For $p \nmid N$ choose $\phi_p$ to be a spherical vector; for a $p |N$ choose $\phi_p$ to be a vector right invariant under $\Gamma_{2n}(p^{m_p})$; and for $p = \infty$ choose  $\phi_\infty$ to be a vector in $\pi_\infty$ spanning the $K_\infty$-type $\rho_k$; see Lemma \ref{scalarKtypeslemma}. Let $f = \otimes f_p \in I(\chi,s)$ be composed of the following local sections. For a finite prime $p \nmid N$ let $f_p$ be the spherical vector normalized by $f_p(1) = 1$; for $p|N$ choose $f_p$ as in Sect.~\ref{s:badplaces} (with the positive integer $m$ of that section equal to the $m_p$ above); and for $p=\infty$, choose $f_\infty$ by \eqref{degprincserhollemmaeq3c}. Define $L^{N}(s,\pi \boxtimes \chi, \varrho_{2n+1}) = \prod_{\substack{ p \nmid N \\ p\ne \infty}} L(s,\pi_p \boxtimes \chi_p, \varrho_{2n+1})$, where the local factors on the right are given by \eqref{Lspichidefeq}.

\begin{theorem}\label{global-thm}
 Let the notation be as above. Then the function $L^N(s,\pi \boxtimes \chi, \varrho_{2n+1})$ can be analytically continued to a meromorphic function of $s$ with only finitely many poles. Furthermore, for all $s \in \C$ and $g\in G_{2n}(\A)$,
 \begin{align}\label{global-int-formula}
  Z(s, f, \phi)(g) =  &\frac{L^N((2n+1)s+1/2,\pi \boxtimes \chi, \varrho_{2n+1})}{L^N((2n+1)(s+1/2), \chi)\prod_{j=1}^nL^N((2n+1)(2s+1)-2j, \chi^2)} \nonumber \\ &\times  i^{nk}\,\pi^{n(n+1)/2}  \bigg(\prod\limits_{p|N}{\rm vol}(\Gamma_{2n}(p^{m_p}))\bigg) A_{\mathbf{k}}((2n+1)s- 1/2)\phi(g),
 \end{align}
 with the  rational function $A_{\mathbf{k}}(z)$ as defined as in Proposition \ref{archzetaprop}.
\end{theorem}
\begin{proof}
By Theorem \ref{basicidentity}, Proposition \ref{unramifiedcalculationprop}, Proposition \ref{prop:badplaces} and Proposition \ref{archzetaprop}, the equation \eqref{global-int-formula} is true for all $\Re(s)$ sufficiently large. Since the left side defines a meromorphic function of $s$ for each $g$, it follows that the right side can be analytically continued to a meromorphic function of $s$ such that \eqref{global-int-formula} always holds.
\end{proof}
\subsection{A classical reformulation}
We now rewrite the above theorem in classical language.  For any congruence subgroup $\Gamma$ of $\Sp_{2n}(\R)$ with the symmetry property $\mat{-I_n}{}{}{I_n} \Gamma \mat{-I_n}{}{}{I_n} = \Gamma$, let $C^\infty_{k}(\Gamma)$ be the space of smooth functions $F:\:\H_n \to \C$ satisfying
\begin{equation}\label{modformeq1}
 F(\gamma Z)=j(\gamma, Z)^k F(Z)\qquad\text{for all }\gamma\in\Gamma.
\end{equation}
For any $F \in C^\infty_{k}(\Gamma)$, there is an element $\bar{F} \in C^\infty_{k}(\Gamma)$ defined via $\bar{F}(Z) = \overline{F(-\overline{Z})}$. Given functions $F_1$, $F_2$ in $C^\infty_{k}(\Gamma)$, we define the Petersson inner product $\langle F_1, F_2 \rangle$ by
$$
 \langle F_1, F_2 \rangle = \vl(\Gamma \bs \H_n)^{-1}\int\limits_{\Gamma \bs \H_n} \det(Y)^k F_1(Z) \overline{F_2(Z)}\,dZ,
$$
whenever the integral above converges. Above, $dZ$ is any $\Sp_{2n}(\R)$-invariant measure on $\H_n$ (it is equal to $c\det(Y)^{-(n+1)}\,dX\,dY$ for some constant $c$). Note that our definition of the  Petersson inner product does not depend on the normalization of measure (the choice of $c$), and is also not affected by different choices of $\Gamma$. Note also that \begin{equation}\label{innerconj}\langle F_1, F_2 \rangle = \overline{\langle \bar{F_1}, \bar{F_2} \rangle}.\end{equation}

Now, let $\Phi$ be an automorphic form on $G_{2n}(\A)$ such that $\Phi(gh)=j(h,I)^{-k}\Phi(g)$ for all $h\in K_\infty^{(n)}\cong U(n)$. Then we can define a function $F(Z)$ on the Siegel upper half space $\mathbb{H}_n$ by
\begin{equation}\label{Hndescenteq}
 F(Z)=j(g,I)^k\,\Phi(g),
\end{equation}
where $g$ is any element of $\Sp_{2n}(\R)$ with $g(I)=Z$. If $\Gamma_p$ is an open-compact subgroup of $G_{2n}(\Q_p)$ such that $\Phi$ is right invariant under $\Gamma_p$ for all $p$, then it is easy to check that $F \in C^\infty_k(\Gamma)$ where $\Gamma=\Sp_{2n}(\Q)\cap\prod_{p<\infty}\Gamma_p$.

We apply this principle to our Eisenstein series $E(g,s,f)$, where $f$ is the global section constructed in Sect.~\ref{global-int-section-main}. Consider
\begin{equation}\label{classical-Eis-ser-1}
 E_{k,N}^{\chi}(Z,s):=j(g, I)^{k} E\Big(g,\frac{2s}{2n+1}+\frac k{2n+1}-\frac 12, f\Big),
\end{equation}
where $g$ is any element of $\Sp_{4n}(\R)$ with $g(I)=Z$. Since the series defining $E(g,s,f)$ converges absolutely for $\Re(s) > \frac12$, it follows that the series defining $E_{k,N}^{\chi}(Z,s)$  converges absolutely whenever $2\Re(s)+k >2n+1$. By the invariance properties of our local sections and by analytic continuation, it follows that for all $s$ we have $E_{k,N}^{\chi}(Z,s) \in C^\infty_k(\Gamma_{4n}(N))$.

It is instructive to write down the function $E_{k,N}^{\chi}(Z,s)$ classically. Indeed, using the relevant definitions, one can check that for $2\Re(s)+k >2n+1$,
\begin{align}\label{eisensteinserieseq}
 E_{k,N}^{\chi}(Z,s) = \sum_{\gamma = \mat{A}{B}{C}{D} \in (P(\Q)\cap \Sp_{4n}(\Z)) \bs \Sp_{4n}(\Z) } \Big(\prod_{p|N} f_p(\gamma, s) \Big) \det(\Im(\gamma Z))^s \det(CZ+D)^{-k}.
\end{align}
Considering the definition of the functions $f_p(\gamma_p, s)$ in Sect.~\ref{s:badplaces}, it is not difficult to see that
\begin{align}\label{eisensteinserieseq2}
 E_{k,N}^{\chi}(Z,s)=\sum_{\gamma\in (Q_n^{-1}P'(\Z)Q_n \cap\Gamma_{4n}(N))\bs \Gamma_{4n}(N)}\det(\Im(Q_n\gamma Z))^s\,j(Q_n\gamma,Z)^{-k}.
\end{align}

Next, let $F(Z)$ be the function on $\H_n$ corresponding to the automorphic form $\phi$ via \eqref{Hndescenteq}.  Then $F, \bar{F} \in C_k^\infty(\Gamma_{2n}(N))$ and both these functions are rapidly decreasing at infinity. For $Z_1,Z_2\in\H_n$, write $E_{k,N}^{\chi}(Z_1,Z_2,s)$ for $E_{k,N}^{\chi}(\mat{Z_1}{}{}{Z_2},s)$.
Using adelic-to-classical arguments similar to Theorem 6.5.1 of \cite{pullback}, we can now write down the classical analogue of Theorem \ref{global-thm}.
\begin{theorem}\label{classical-integral-repn}
Let $E_{k,N}^{\chi}(Z,s)$ be as defined in \eqref{classical-Eis-ser-1}, and let $F(Z)$, $\bar{F}(Z)$ be as defined above. Then we have the relation
\begin{align*}
  \langle E_{k,N}^{\chi}( \ - \ , Z_2, \frac{n}2 - \frac{k-s}{2}), \bar{F} \rangle &=  \frac{L^N(s,\pi \boxtimes \chi, \varrho_{2n+1})}{L^N(s+n, \chi)\prod_{j=1}^nL^N(2s+2j-2, \chi^2)} \times A_{\mathbf{k}}(s-1)  \\
  & \times \prod_{p|N}{\rm vol}(\Gamma_{2n}(p^{m_p})) \times \frac{ i^{nk}\,\pi^{n(n+1)/2} }{\vol(\Sp_{2n}(\Z) \bs \Sp_{2n}(\R))} \times F(Z_2).
 \end{align*}
\end{theorem}

\begin{corollary}\label{criticalcor}Let $r$ be a positive integer satisfying $1\le r \le  k_n-n$, $r \equiv k_n -n \pmod{2}$. Then

 $$ \left \langle E_{k,N}^{\chi}\left(\ - \ , Z_2, \frac{n}2- \frac{k-r}2\right) ,\bar{F} \right \rangle
=  i^{nk} c_{k,r,n,N}  \frac{L^N(r,\pi \boxtimes \chi, \varrho_{2n+1})}{L^N(r+n, \chi)\prod_{j=1}^nL^N(2r+2j-2, \chi^2)} F(Z_2) $$
  where
 $c_{k,r,n,N} = \prod_{p|N}{\rm vol}(\Gamma_{2n}(p^{m_p})) \times \frac{ \pi^{n(n+1)/2} }{\vol(\Sp_{2n}(\Z) \bs \Sp_{2n}(\R))} \times A_{\mathbf{k}}(r-1)$ is a non-zero rational number.

\end{corollary}

We end this section with an important property of the Eisenstein series $E_{k,N}^{\chi}(Z, \frac{n}2- \frac{k-r}2)$ appearing in Corollary \ref{criticalcor}.  For each positive integer $n$, let $N(\H_{n})$ be the space of nearly holomorphic functions on $\H_n$. By definition, these are the functions which are polynomials in the entries of ${\rm Im}(Z)^{-1}$ with holomorphic functions on $\H_n$ as coefficients. For each discrete subgroup $\Gamma$ of $\Sp_{2n}(\R)$, let $N_k(\Gamma)$ be the space of all functions $F$ in $N(\H_n)$ that satisfy $F(\gamma Z) = j(\gamma, Z)^k F(Z)$ for all $Z \in \H_n$ and $\gamma \in \Gamma$  (if $n=1$, we also need an additional ``no poles at cusps" condition, as explained in \cite{PSS14}). The spaces $N_k(\Gamma)$ are known as nearly holomorphic modular forms of weight $k$ for $\Gamma$; a detailed explication of these spaces for $n=2$ can be found in \cite{PSS14}.

\begin{proposition}\label{propnearholo} Suppose that $k\ge 2n+2$. Then the series defining $E_{k,N}^{\chi}(Z,0)$ is absolutely convergent and defines a holomorphic Siegel modular form of degree $2n$ and weight $k$ with respect to the principal congruence subgroup $\Gamma_{4n}(N)$ of $\Sp_{4n}(\R)$. More generally, let $0 \le m \le \frac{k}2-n-1$ be an integer. Then $E_{k,N}^{\chi}(Z, -m) \in N_k(\Gamma_{4n}(N))$.
\end{proposition}
\begin{proof} We already noted above that $E_{k,N}^{\chi}(Z,s)$ has the transformation and growth properties of a Siegel modular form of weight $k$ and that the series \eqref{eisensteinserieseq} defining $E_{k,N}^{\chi}(Z, s)$ converges absolutely whenever $2 \Re(s)+ k >2n+1$. So to complete the proof, it suffices to show that for each $\gamma \in\Sp_{4n}(\Q)$, and each integer $j$, we have $\det(\Im(\gamma Z))^{-j} \in N(\H_{2n})$. This is immediate from the well-known equation
$$
 \Im(\gamma Z)^{-1} = (CZ+D)Y^{-1}\:^t(CZ +D) - 2i (CZ+D)\,^tC,
$$
valid for all $\gamma = \mat{A}{B}{C}{D}$ in $\Sp_{4n}(\R)$.
\end{proof}
\begin{remark}
It should be possible to replace the hypothesis $k\ge 2n+2$ with $k \ge n+1$ and the hypothesis $0 \le m \le \frac{k}2-n-1$ with $0 \le m \le \frac{k}2-\frac{n+1}2$, using more refined arguments as in the work of Shimura \cite{shibook2} and others.
\end{remark}
\section{Nearly holomorphic modular forms and critical \texorpdfstring{$L$}{}-values}
\subsection{Preliminaries}\label{final:prelim}
Throughout this section we assume $n=2$. Let $\ell, m$ be non-negative integers with $m$ even and $\ell \ge 6$. We put ${\mathbf{k}}=(\ell+m)e_1+\ell e_2$ and  $k = \ell +m$. For each integer $N= \prod_p p^{m_p}$, we let $\Pi_N(\mathbf{k})$ denote the set of all  cuspidal automorphic representations $\pi\cong\otimes \pi_p$ of $G_{4}(\A)$ such that $\pi_\infty$ equals the holomorphic discrete series representation $\pi_{\mathbf{k}}$ and such that for each finite prime $p$, $\pi_p$ has a vector right invariant under the principal congruence subgroup $\Gamma_{4}(p^{m_p})$ of $\Sp_{4}(\Z_p)$. We put $\Pi(\mathbf{k}) = \bigcup_N\Pi_N(\mathbf{k})$.

We say that a character $\chi=\otimes\chi_p$ of $\Q^\times \bs
\A^\times$ is a \emph{Dirichlet character} if $\chi_\infty$ is trivial
on $\R_{>0}$. Any such $\chi$ gives rise to a homomorphism $\tilde{\chi}:(\Z/N_\chi)^\times\rightarrow \C^\times$, where $N_\chi$ denotes the conductor of $\chi$.
Concretely\footnote{In fact, the map $\chi \mapsto \tilde{\chi}$ gives a
bijection between Dirichlet characters in our sense and primitive
Dirichlet characters in the classical sense.} the map $\tilde{\chi}$ is
given by $\tilde{\chi}(a) = \prod_{p|N_\chi} \chi_p(a)$. Given a
Dirichlet character $\chi$, we define the corresponding Gauss sum by
$G(\chi) = \sum_{n \in  (\Z/N_\chi)^\times }\tilde{\chi}(n) e^{2 \pi i
n/N_\chi}.$

Let $\mathcal{X}_N$  be the set of all Dirichlet characters $\chi$ such that $\chi_\infty = \sgn^\ell$ and  $N_\chi|N$. We denote $\mathcal{X} = \bigcup_N \mathcal{X}_N$. For any $\pi \in \Pi_N(\mathbf{k})$, $\chi \in \mathcal{X}_N$,  and any integer $r$, we define, following the notation of Corollary \ref{criticalcor},
$$
 C_N(\pi, \chi, r):=  (-1)^k \pi^{2r+4-2k} \ c_{k,r,2,N}  \ \frac{L^N(r,\pi \boxtimes \chi, \varrho_{5})}{L^N(r+2, \chi) L^N(2r, \chi^2) L^N(2r+2, \chi^2)}.
$$
The reader should not be confused by the two different $\pi$ (one the constant, the other an automorphic representation) appearing in the above definition. From the results in \cite[p. 236]{dhl92}
it follows that  if $r> 2$, we are in the absolutely convergent range for all the relevant $L$-functions and therefore $C_N(\pi, \chi, r)$ is not equal to either 0 or $\infty$.

Recall that $N_k(\Gamma_4(N))$ denotes the (finite-dimensional) space of nearly holomorphic modular forms of weight $k$ for the subgroup $\Gamma_4(N)$ of $\Sp_4(\Z)$. Let $V_N$ be the subset of $N_k(\Gamma_4(N))$ consisting of those forms $F$ which are \emph{cuspidal} and for which the corresponding function $\Phi_F$ on $\Sp_4(\R)$ generates an irreducible representation isomorphic to $\pi_{\mathbf{k}}$. By Theorem 4.8 and Proposition 4.28 of \cite{PSS14}, $V_N$ is a \emph{subspace} of $N_k(\Gamma_4(N))$ and isomorphic to the space $S_{\ell, m}(\Gamma_4(N))$ of holomorphic vector-valued cusp forms of weight $\det^\ell \sym^m$ for $\Gamma_4(N)$; indeed $V_N=U^{m/2}(S_{\ell, m}(\Gamma_4(N)))$, where $U$ is the differential operator defined in Section 3.4 of \cite{PSS14}.  We put $V= \bigcup_N V_N$ and $N_k = \bigcup_N N_k(\Gamma_4(N))$.

As in \cite{PSS14}, we let $\mathfrak{p}^\circ_{\ell, m}$ denote the orthogonal projection map from $N_k$ to $V$; note that it takes $N_k(\Gamma_4(N))$ to $V_N$ for each $N$. Let $4 \le r \le \ell-2$, $r \equiv \ell \pmod{2}$ be an integer and $\chi \in \mathcal{X}_N$. Then $E_{k,N}^\chi(Z_1, Z_2, 1 - \frac{k-r}2) \in N_k(\Gamma_4(N)) \otimes N_k(\Gamma_4(N))$. We define
\begin{equation}\label{GkNdefeq}
 G_{k,N}^\chi(Z_1, Z_2, r) := \pi^{2r+4 -2k} \times (\mathfrak{p}^\circ_{\ell, m} \otimes  \mathfrak{p}^\circ_{\ell, m}) (E_{k,N}^\chi(Z_1, Z_2, 1 - \frac{k-r}2)).
\end{equation}

If $F \in V$ is such that (the adelization of) $F$ generates a multiple of an irreducible (cuspidal, automorphic) representation of $G_4(\A)$, then we let $\pi_F$ denote the representation associated to $F$. Note that the set of automorphic representations $\pi_F$ obtained this way as $F$ varies in $V_N$ is precisely equal to the set $\Pi_N(\mathbf{k})$ defined above. For each $\pi \in \Pi(\mathbf{k})$, we let $V_N(\pi)$ denote the $\pi$-isotypic part of $V_N$. Precisely, this is the subspace consisting of all those $F$ in $V_N$ such that all irreducible constituents of the representation generated by (the adelization of) $F$ are isomorphic to $\pi$. Note that $V_N(\pi)= \{0\}$ unless $\pi \in \Pi_N(\mathbf{k})$. We have an orthogonal direct sum decomposition \begin{equation}\label{E:VN}V_N = \bigoplus_{\pi \in \Pi_N(\mathbf{k})} V_N(\pi).\end{equation}

We define $V(\pi) = \bigcup_N V_N(\pi)$. Therefore we have $V = \bigoplus_{\pi \in \Pi(\mathbf{k})} V(\pi)$. Now, let $\mathfrak{B}$ be \emph{any} orthogonal basis of $V_N$ formed by taking a union of orthogonal bases from the right side of \eqref{E:VN}. Thus each $F \in \mathfrak{B}$ belongs to $V_N(\pi)$ for some $\pi \in \Pi_N(\mathbf{k})$. From Corollary \ref{criticalcor}, Proposition \ref{propnearholo}, and \eqref{innerconj}, we deduce the following key identity:

\begin{align}\label{keyidentity}G_{k,N}^\chi(Z_1, Z_2, r) &= \sum_{F \in \mathfrak{B}} C_N(\pi_F, \chi, r) \frac{F(Z_1) \bar{F}(Z_2)}{\langle F, F\rangle} \nonumber \\ &= \sum_{\pi \in \Pi_N(\mathbf{k})} C_N(\pi, \chi, r) \sum_{\substack{F \in \mathfrak{B}\\ F \in V_N(\pi)}} \frac{F(Z_1) \bar{F}(Z_2)}{\langle F, F\rangle}.  \end{align}
\subsection{Arithmeticity}
For any nearly holomorphic modular form $F \in N_k(\Gamma)$ (where $\Gamma$ is a congruence subgroup of $\Sp_{4}(\Z)$), we let ${}^\sigma F$ denote the nearly holomorphic modular form obtained by letting $\sigma$ act on the Fourier coefficients of $F$ (see Section 5.2 of \cite{PSS14}). Note that if $\sigma$ is complex conjugation  then ${}^\sigma F$ equals $\bar{F}$.

\begin{proposition}
Let $0 \le m \le \frac{k}2 - 3$, and $\chi \in \mathcal{X}_N$. Then for any $\sigma \in \Aut(\C)$, we have
$$
 {}^\sigma E_{k,N}^\chi(Z, -m) = \frac{\sigma(\pi^{4m})}{\pi^{4m}} E_{k,N}^{{}^\sigma\!\chi}(Z, -m).
$$
\end{proposition}
 \begin{proof}
Recall that $E_{k,N}^\chi(Z, 0)$ is a holomorphic Eisenstein series in the region of absolute convergence. So in the case $m=0$, the result follows from the work of Harris (the main result of \cite{hareis}). For the case $m>0$, we need the Maass-Shimura differential operator $\Delta_k = \det(Y)^{-1} M_k$, where the operator $M_k$ was first defined by Maass \cite[\S 19]{maassbook}.  Maass showed that for any $\gamma \in \Sp_8(\R)$, we have
$$
 \Delta_k\left(\det(\Im(\gamma Z))^s j(\gamma, z)^{-k}\right) = d_k\det(\Im(\gamma Z))^{s-1} j(\gamma, z)^{-k-2},
$$
where $d_k$ is a non-zero rational number. Using the absolutely convergent expression \eqref{eisensteinserieseq}, we deduce  that $\Delta_k(E_{k,N}^\chi(Z, -m+1)) =  d_k E_{k+2,N}^\chi(Z, -m)$. It follows that $E_{k,N}^\chi(Z, -m)$ is a rational multiple of $\Delta_{k-2} \circ \ldots \circ\Delta_{k-2m}E_{k-2m,N}^\chi(Z,0)$.  The result now follows by induction from the $m=0$ case and the identity
\begin{equation}\label{e:massoparith}
 {}^\sigma(\pi^{-4}\Delta_k F) = \Delta_k(\pi^{-4} \ {}^\sigma\!F),
\end{equation}
which holds for all nearly holomorphic modular forms $F$ of degree 4. The identity \eqref{e:massoparith} is an immediate and elementary consequence of the explicit formula for $\Delta_k$ (see the proof of \cite[Lemma 8]{sturm}); alternatively, one can use geometric arguments as in Corollary 6.9 of \cite{harsieg}.
 \end{proof}

\begin{corollary}\label{eisequivarcor}
For all $\sigma \in \Aut(\C)$ and $4 \leq r \leq \ell-2, r \equiv \ell \pmod{2}$, we have
$$
 {}^\sigma G_{k,N}^\chi(Z_1, Z_2, r) = G_{k,N}^{{}^\sigma\! \chi}(Z_1, Z_2, r).
$$
In particular, the Fourier coefficients of $G_{k,N}^\chi(Z_1, Z_2, r)$ lie in $\Q(\chi)$.
\end{corollary}
\begin{proof}
Recall that $G_{k,N}^\chi(Z_1, Z_2, r)$ is defined by \eqref{GkNdefeq}. So the corollary follows from the above proposition and the fact that the map $\mathfrak{p}^\circ_{\ell, m}$ commutes with the action of $\Aut(\C)$ (see Proposition 5.17 of \cite{PSS14}). A key point here is that because we are in the absolutely convergent range $\ell\ge 6$ and as $\Gamma_4(N)$ is a principal congruence subgroup, the orthogonal complement to $S_{\ell, m}(\Gamma_4(N))$ is spanned by holomorphic vector-valued Eisenstein series, and by the main result of \cite{hareis}, the action of any element of $\Aut(\C)$ on such an Eisenstein series just gives another Eisenstein series. In particular the  necessary $\Q$-rational condition in Proposition 5.17 of \cite{PSS14} holds.
\end{proof}

Let $\sigma \in \Aut(\C)$. For $\pi \in \Pi_N(\mathbf{k})$, we let ${}^\sigma \pi \in \Pi_N(\mathbf{k})$  be the representation obtained by the action of $\sigma$, and we let $\Q(\pi)$ denote the field of rationality of $\pi$; see the beginning of Section 3.4 of \cite{sahapet}. If $\sigma$ is the complex conjugation, then we denote ${}^\sigma \pi = \bar{\pi}$. It is known that $\Q(\pi)$ is a CM field and $\Q(\bar{\pi}) = \Q(\pi)$. We use $\Q(\pi, \chi)$ to denote the compositum of $\Q(\pi)$ and $\Q(\chi)$.  Note that
\begin{equation}\label{FVNEQ}
 F \in V_N(\pi) \;\Longrightarrow\; {}^\sigma F \in V_N({}^\sigma\pi).
\end{equation}
This follows from Theorem 4.2.3 of \cite{blaharam} (see the proof of  Proposition 3.13 of \cite{sahapet}) together with the fact that the $U$ operator commutes with $\sigma$. In particular, the space $V_N(\pi)$ is preserved under the action of the group $\Aut(\C/\Q(\pi))$. Using Lemma 3.17 of \cite{sahapet}, it follows that the space $V_N(\pi)$ has a basis consisting of forms whose Fourier coefficients are in $\Q(\pi)$. In particular there exists some $F$ satisfying the conditions of the next proposition.

\begin{proposition}\label{mainprop}
 Let $\pi \in \Pi_N(\mathbf{k})$, and $F \in V_N(\pi)$. Suppose that the Fourier coefficients of $F$ lie in a CM field. Then for any $4 \le r \le \ell-2$, $r \equiv \ell \pmod{2}$, and any $\sigma \in \Aut(\C)$ we have
 $$
  \sigma \left(\frac{C_N(\pi, \chi, r)}{\langle F, F\rangle} \right) = \frac{C_N({}^\sigma\pi, {}^\sigma\!\chi, r)}{\langle {}^\sigma\! F, {}^\sigma\!F\rangle}.
 $$
\end{proposition}

\begin{proof}
Let us complete $F$ to an orthogonal basis $\mathfrak{B}=\{F=F_1, F_2, \ldots, F_r \}$ of $V_N(\pi)$.  Let $\mathfrak{B}'=\{G_1, \ldots, G_r \}$ be any orthogonal basis for $V_N({}^\sigma\pi)$. Using \eqref{keyidentity}, \eqref{FVNEQ}, and Corollary \ref{eisequivarcor}, and comparing the $V_N({}^\sigma \pi)$ components, we see that
$$
 \sigma (C_N(\pi, \chi, r)) \sum_{i} \frac{{}^\sigma\!F_i(Z_1) {}^\sigma\!\bar{F_i}(Z_2)}{\sigma(\langle F_i, F_i\rangle)} =C_N({}^\sigma\!\pi, {}^\sigma\!\chi, r) \sum_{i} \frac{G_i(Z_1) \bar{G_i}(Z_2)}{\langle G_i, G_i\rangle}.
$$
Taking inner products of each side with ${}^\sigma\!F_1$ (in the variable $Z_1$) we deduce that
$$
 \sigma (C_N(\pi, \chi, r)) \sum_{i} \frac{\langle {}^\sigma\! F_i, {}^\sigma\! F_1 \rangle }{\sigma(\langle F_i, F_i\rangle)} {}^\sigma\bar{F_i}(Z_2) =C_N({}^\sigma\!\pi, {}^\sigma\!\chi, r)\,\overline{{}^\sigma\!F_1} (Z_2).
$$
Noting that $\overline{{}^\sigma\!F_1} = {}^\sigma\!\bar{F_1}$  and comparing the coefficients of ${}^\sigma\!\bar{F_1}(Z_2)$ on each side, we conclude the desired equality.
\end{proof}
\subsection{Critical \texorpdfstring{$L$}{}-values}For each $p|N$, we define the local $L$-factor $L(s,\pi_p \boxtimes \chi_p, \varrho_{5})$ via the local Langlands correspondence \cite{gantakGSp4}. In particular, $L(s,\pi_p \boxtimes \chi_p, \varrho_{5})$ is just a local $L$-factor for $\GL(5) \times\GL(1)$. This definition also works at the good places, and indeed coincides with what we previously defined. For any finite set of places $S$ of $\Q$, we define the global $L$-function $$L^S(r,\pi \boxtimes \chi, \varrho_{5}) = \prod_{p \notin S} L(s,\pi_p \boxtimes \chi_p, \varrho_{5}).$$

\begin{theorem}\label{t:globalthmarithmetic}
 Let $\pi \in \Pi(\mathbf{k})$, and  let $F\in V(\pi)$ be such that its Fourier coefficients  lie in a CM field. Let $S$ be any finite set of places of $\Q$ containing the infinite place. Then, for any $\chi \in \mathcal{X}$, any integer $r$ satisfying $4 \le r \le \ell-2$, $r \equiv \ell \pmod{2}$,  and any $\sigma \in \Aut(\C)$, we have
 \begin{equation}\label{t:globalthmarithmeticeq1}
  \sigma \left(\frac{L^S(r,\pi \boxtimes \chi, \varrho_{5})}{(2\pi i)^{2k+3r} G(\chi) G(\chi^2)^2 \langle F, F\rangle}\right) = \frac{L^S(r,{}^\sigma\pi \boxtimes {}^\sigma\!\chi, \varrho_{5})}{(2\pi i)^{2k+3r} G({}^\sigma\!\chi) G({}^\sigma\!\chi^2)^2 \langle {}^\sigma\!F,{}^\sigma\!F\rangle}.
 \end{equation}
\end{theorem}
\begin{proof}  Let $\chi \in \mathcal{X}$. We  choose some $N$ such that $F \in V_N(\pi)$ and  $\chi \in \mathcal{X}_N$.
By Proposition \ref{mainprop},
\begin{equation}\label{finaleq1}
 \begin{split}
  &\sigma \left( \frac{L^N(r,\pi \boxtimes \chi, \varrho_{5})}{\pi^{2k-2r-4} L^N(r+2, \chi) L^N(2r, \chi^2) L^N(2r+2, \chi^2)\langle F, F\rangle} \right) \\ &= \frac{L^N(r,{}^\sigma\pi \boxtimes {}^\sigma\!\chi, \varrho_{5})}{\pi^{2k-2r-4} L^N(r+2, {}^\sigma\!\chi) L^N(2r, {}^\sigma\!\chi^2) L^N(2r+2, {}^\sigma\!\chi^2)\langle {}^\sigma\!F, {}^\sigma\!F\rangle}.
 \end{split}
\end{equation}
Next, for any finite prime $p$ such that $p\in S$, $p|N$, and any Dirichlet character $\psi$, it follows from Lemma 4.6 of \cite{clozel} that
\begin{equation}\label{finaleq2}
 \sigma( L(r,\pi_p \boxtimes \chi_p, \varrho_{5})) = L(r,{}^\sigma\pi_p \boxtimes {}^\sigma\! \chi_p, \varrho_{5}), \quad \sigma( L(r, \psi_p )) = L(r,{}^\sigma\!\psi_p).
\end{equation}
Finally, for any Dirichlet character $\psi$, we have the following  fact (see \cite[Lemma 5]{shi76}) for any positive integer $t$ satisfying $\psi(-1) = (-1)^t$:
\begin{equation}\label{finaleq3}
 \sigma \left(\frac{L^N(t, \psi)}{(2\pi i)^t G(\psi)} \right) = \frac{L^N(t, {}^\sigma\!\psi)}{(2\pi i)^t G({}^\sigma\!\psi)}.
\end{equation}
Combining \eqref{finaleq1}, \eqref{finaleq2}, \eqref{finaleq3} (with $\psi = \chi$ and $\chi^2$) we get the desired result.
\end{proof}

\begin{remark}Let $F$ be as in the above theorem. By the results of \cite{PSS14}, we know that $F = U^{m/2} F_0$ where $F_0$ is a holomorphic vector-valued Siegel cusp form whose adelization generates a multiple of $\pi$. Using Lemma 4.16 of \cite{PSS14}, we have moreover the equality $\langle F_0, F_0 \rangle = c_{m} \langle F, F \rangle$ for some constant $c_m$. A calculation (that we do not perform here) shows that $c_m$ is a rational multiple of $\pi^m$. Hence in the theorem above, the term $\langle F, F\rangle$ can be replaced by  $\pi^{-m}\langle F_0, F_0 \rangle$.
\end{remark}

\begin{definition}For two representations $\pi_1$, $\pi_2$ in $\Pi(\mathbf{k})$, we write $\pi_1 \approx \pi_2$ if there is a Hecke character $\psi$ of $\Q^\times \bs \A^\times$ such that $\pi_1$ is nearly equivalent to $\pi_2 \otimes \psi$.
\end{definition}

Note that if such a $\psi$ as above exists, then $\psi_\infty$ must be trivial on $\R_{>0}$ and therefore $\psi$ must be a Dirichlet character. The relation $\approx$ clearly gives an equivalence relation on $\Pi(\mathbf{k})$. For any $\pi \in \Pi(\mathbf{k})$, let $[\pi]$ denote the class of $\pi$, i.e., the set of all representations $\pi_0$ in   $\Pi(\mathbf{k})$ satisfying $\pi_0 \approx \pi$. For any integer $N$, we define the subspace $V_N([\pi])$ of $V_N$ to be the (direct) sum of all the subspaces $V_N(\pi_0)$ where $\pi_0$ ranges over all the inequivalent representations in $[\pi]\cap \Pi_N(\mathbf{k})$.

\begin{corollary}\label{cor:petersson}
 Let $\pi_1, \pi_2 \in \Pi(\mathbf{k})$ be such that $\pi_1 \approx \pi_2$. Let $F_1 \in V(\pi_1)$ and $F_2 \in V(\pi_2)$ have coefficients in a CM field. Then for all $\sigma \in \Aut(\C)$, we have
 $$
  \sigma \left(\frac{\langle F_1, F_1 \rangle}{\langle F_2, F_2 \rangle} \right) = \frac{\langle {}^\sigma F_1, {}^\sigma F_1 \rangle}{\langle {}^\sigma F_2, {}^\sigma F_2 \rangle}.
 $$
\end{corollary}
\begin{proof}
By assumption, there is a character $\psi$ and a set $S$ of places containing the infinite place, such that $\pi_{1,p} \simeq \pi_{2,p} \otimes \psi_p$ for all $p \notin S$. We fix any character $\chi\in \mathcal{X}$. Note that $L(s, \pi_{1,p} \boxtimes \chi_p,\varrho_5) =  L(s, \pi_{1,p} \boxtimes \chi_p,\varrho_5)$ for all $p \notin S$, as the representation $\varrho_5$ factors through $\PGSp_4$ and therefore is blind to twisting by $\psi$.   Applying Theorem \ref{t:globalthmarithmetic} twice at the point $r= \ell-2$, first with $(\pi_1, F_1)$, and then with $(\pi_2$, $F_2)$, and dividing the two equalities, we get the desired result.
\end{proof}

We can now prove Theorem \ref{t:mainintro}. For each $\pi \in \Pi(\mathbf{k})$ we need to define a quantity $C(\pi)$ that has the properties claimed in Theorem \ref{t:mainintro}. The first two properties required by Theorem \ref{t:mainintro} can be summarized as saying that $C(\pi)$ depends only on the class $[\pi]$ of $\pi$. For each $\pi \in \Pi(\mathbf{k})$,  we let $N([\pi])$ denote the smallest integer such that $[\pi] \cap \Pi_{N([\pi])}(\mathbf{k}) \neq \emptyset.$ Next, note that if $\pi_1 \approx \pi_2$, then for any $\sigma \in \Aut(\C)$, ${}^\sigma\pi_1 \approx {}^\sigma\pi_2$. Thus for any class $[\pi]$ and any $\sigma \in \Aut(\C)$, we have a well-defined notion of the class ${}^\sigma[\pi]=[{}^\sigma\pi]$. Let $\Q([\pi])$ denote the fixed field of the set of all automorphisms of $\C$ such that  ${}^\sigma[\pi] = [\pi]$. Clearly the field $\Q([\pi])$ is contained in $\Q(\pi_0)$ for any $\pi_0 \in [\pi]$. It is also clear that $N({}^\sigma[\pi]) = N([\pi])$ and $\Q({}^\sigma[\pi]) = \sigma (\Q([\pi]))$ for any $\sigma \in \Aut(\C)$.

  The space $V_{N([\pi])}([\pi])$ is preserved under the action of $\Aut(\C/\Q([\pi]))$ and therefore has a basis consisting of forms with coefficients in   $\Q([\pi])$. Suppose that $G$ is a non-zero element of $V_{N([\pi])}([\pi])$ whose Fourier coefficients lie in $\Q([\pi])$. Then for any $\sigma \in \Aut(\C)$, ${}^\sigma G$ is an element of $V_{N({}^\sigma[\pi])}({}^\sigma[\pi])$ whose Fourier coefficients lie in $\Q({}^\sigma[\pi])$. Moreover, if $\sigma, \tau$ are elements of $\Aut(\C)$ such that ${}^\sigma[\pi] =  {}^\tau[\pi]$, then ${}^\sigma G =  {}^\tau G.$ It follows that for each class $[\pi]$, we can choose a non-zero element $G_{[\pi]} \in V_{N([\pi])}([\pi])$ such that the following two properties hold:
  \begin{enumerate}
  \item The Fourier coefficients of $G_{[\pi]}$ lie in $\Q([\pi])$.
  \item $G_{{}^\sigma[\pi]} = {}^\sigma G_{[\pi]}.$
  \end{enumerate}
  We do \emph{not} require that the adelization of $G_{[\pi]}$ should generate an irreducible representation.

  Finally, for any $\pi \in \Pi(\mathbf{k})$ we define $$C(\pi) = (2 \pi i)^{2k} \langle G_{[\pi]}, G_{[\pi]} \rangle. $$ By construction, $C(\pi)$ depends only on the class $[\pi]$ of $\pi$. So we only need to prove \eqref{e:mainarith}. The following lemma is key.

\begin{lemma}\label{finallemma}
 Let $\pi, F$ be as in Theorem \ref{t:globalthmarithmetic}. Then for all $\sigma \in \Aut(\C)$,
 $$
  \sigma \left(\frac{\langle F, F \rangle}{\langle G_{[\pi]}, G_{[\pi]} \rangle} \right) = \frac{\langle\,{}^\sigma\!F, {}^\sigma\!F \rangle}{\langle {}^\sigma G_{[\pi]}, {}^\sigma G_{[\pi]} \rangle}.
 $$
\end{lemma}
\begin{proof}
Write $G_{[\pi]} = F_1 + F_2 + \ldots+F_t$ where $F_i \in V(\pi_i)$ for inequivalent representations $\pi_i \in [\pi]$. Note that the spaces $V(\pi_i)$ are mutually orthogonal and hence
$$
 \langle G_{[\pi]}, G_{[\pi]} \rangle = \sum_{i=1}^t \langle F_i, F_i \rangle, \quad \langle\, {}^\sigma G_{[\pi]}, {}^\sigma G_{[\pi]} \rangle =  \sum_{i=1}^t \langle\, {}^\sigma\!F_i, {}^\sigma\!F_i \rangle.
$$
So the desired result would follow immediately from Corollary \ref{cor:petersson} provided we can show that each $F_i$ has coefficients in a CM field.

Indeed, let $K$ be the compositum of all the fields $\Q(\pi_i)$. Thus $ K$ is a CM field containing $\Q([\pi])$. For any $\sigma \in \Aut(\C/K)$, we have ${}^\sigma G_{[\pi]} = G_{[\pi]}$ and ${}^\sigma F_i \in V(\pi_i)$. As the spaces $V(\pi_i)$ are all linearly independent, it follows that ${}^\sigma F_i =  F_i$ and therefore each $F_i$ has coefficients in a CM field.
\end{proof}

The proof of \eqref{e:mainarith} follows by combining Theorem \ref{t:globalthmarithmetic} and Lemma \ref{finallemma}.
\begin{appendix}
\section{Haar measures on \texorpdfstring{$\Sp_{2n}(\R)$}{}}\label{measureapp}
This appendix will furnish proofs for the constants appearing in the integration formulas \eqref{KAKintegrationeq} and \eqref{Sp2niwasawainteq2}. The symbol $K$ denotes the maximal compact subgroup $\Sp_{2n}(\R)\cap{\rm O}(2n)$ of $\Sp_{2n}(\R)$.
\subsection{The \texorpdfstring{$KAK$}{} measure}\label{KAKmeasureapp}
Recall that we have fixed the ``classical'' Haar measure on $\Sp_{2n}(\R)$ characterized by the property \eqref{Ghaarpropeq1}. There is also the ``$KAK$ measure'' given by the integral on the right hand side of \eqref{KAKintegrationeq}.
\begin{proposition}\label{alphanprop}
 The constant $\alpha_n$ in \eqref{KAKintegrationeq} is given by
 \begin{equation}\label{alphanpropeq1}
  \alpha_n=\frac{(4\pi)^{n(n+1)/2}}{\prod_{m=1}^n(m-1)!}.
 \end{equation}
\end{proposition}
\begin{proof}
The proof consists of evaluating both sides of \eqref{KAKintegrationeq} for the function
\begin{equation}\label{alphanpropeq2}
 F(\mat{A}{B}{C}{D})=\frac{2^{2nk}}{|\det{(A+D+i(C-B))|^{2k}}}.
\end{equation}
Note that this function is the square of the absolute value of the matrix coefficient appearing in \eqref{matrixcoeffeq}. Hence the integrals will be convergent as long as $k>n$. The function $f(Z)$ on $\H_n$ corresponding to $F$ is given by
\begin{align}\label{nKAKeq2}
 f(Z)&=F(\mat{1}{X}{}{1}\mat{Y^{1/2}}{}{}{Y^{-1/2}})\nonumber\\
 &=\frac{2^{2nk}}{|\det{(Y^{1/2}+Y^{-1/2}-iXY^{-1/2})|^{2k}}}\nonumber\\
 &=\frac{2^{2nk}\,\det(Y)^k}{|\det{(1_n+Y-iX)|^{2k}}}.
\end{align}
Hence, by \eqref{Ghaarpropeq1},
\begin{equation}\label{nKAKeq3}
 \int\limits_{\Sp_{2n}(\R)}F(g)\,dg=2^{2nk}\int\limits_{\H_n}\frac{\det(Y)^{k-n-1}}{|\det{(1_n+Y-iX)|^{2k}}}\,dX\,dY.
\end{equation}
We now employ the following integral formula. For a matrix $X$, denote by $[X]_p$ the upper left block of size $p\times p$ of $X$. For $j=1,\ldots,n$ let $\lambda_j,\sigma_j,\tau_j$ be complex numbers, and set $\lambda_{n+1}=\sigma_{n+1}=\tau_{n+1}=0$. Then, by (0.11) of \cite{Neretin2000},

\begin{align}\label{neretinformulaeq1}
 &\int\limits_{\H_n}\:\prod_{j=1}^n\frac{\det[Y]_j^{\lambda_j-\lambda_{j+1}}}{\det[1_n+Y+iX]_j^{\sigma_j-\sigma_{j+1}}\det[1_n+Y-iX]_j^{\tau_j-\tau_{j+1}}}\det(Y)^{-(n+1)}\,dX\,dY\nonumber\\
 &\qquad=\prod_{m=1}^n\frac{2^{2-\sigma_m-\tau_m+n-m}\,\pi^m\,\Gamma(\lambda_m-(n+m)/2)\,\Gamma(\sigma_m+\tau_m-\lambda_m-(n-m)/2)}{\Gamma(\sigma_m-(n-m)/2)\,\Gamma(\tau_m-(n-m)/2)},
\end{align}
provided the integral is convergent. We will only need the special case where all $\lambda_j$ are equal to some $\lambda$, all $\sigma_j$ are equal to some $\sigma$, and all $\tau_j$ are equal to some $\tau$. In this case the formula says that
\begin{align}\label{neretinformulaeq2}
 &\int\limits_{\H_n}\frac{\det(Y)^\lambda}{\det(1_n+Y+iX)^\sigma\det(1_n+Y-iX)^\tau}\det(Y)^{-(n+1)}\,dX\,dY\nonumber\\
 &\qquad=2^{-(\sigma+\tau)n+n(n+3)/2}\pi^{n(n+1)/2}\prod_{m=1}^n\frac{\Gamma(\lambda-(n+m)/2)\Gamma(\sigma+\tau-\lambda-(n-m)/2)}{\Gamma(\sigma-(n-m)/2)\Gamma(\tau-(n-m)/2)}.
\end{align}
If we set $\lambda=\sigma=\tau=k$, we obtain precisely the integral \eqref{nKAKeq3}. Hence,
\begin{align}\label{neretinformulaeq3}
 \int\limits_GF(g)\,dg&=2^{n(n+3)/2}\,\pi^{n(n+1)/2}\prod_{m=1}^n\frac{\Gamma(k-(n+m)/2)}{\Gamma(k-(n-m)/2)}\nonumber\\
  &=2^{n(n+2)}\,\pi^{n(n+1)/2}\prod_{m=1}^n\,\prod_{j=1}^m\frac{1}{2k-n-m-2+2j}.
\end{align}
Next we evaluate the right hand side ${\rm RHS}$ of \eqref{KAKintegrationeq}. With the same calculation as in the proof of Proposition \ref{scalarminKtypeprop}, we obtain
\begin{equation}\label{alphanpropeq3}
  {\rm RHS}=\alpha_n2^n\int\limits_T\bigg(\prod_{1\leq i<j\leq n}(t_i^2-t_j^2)\bigg)\Big(\prod_{j=1}^nt_j\Big)^{1-2k}\,d\mathbf{t},
\end{equation}
with $T$ as in \eqref{Blambdaformulaeq4}. Thus, by Lemma \ref{Tintegrallemma},
\begin{equation}\label{alphanpropeq4}
  {\rm RHS}=\alpha_n2^n\prod_{m=1}^n(m-1)!\prod_{j=1}^m\frac{1}{2k-n-m-2+2j}.
\end{equation}
Our assertion now follows by comparing \eqref{neretinformulaeq3} and \eqref{alphanpropeq4}.
\end{proof}
\subsection{The Iwasawa measure}\label{Iwasawameasureapp}
In this section we will prove the integration formula \eqref{Sp2niwasawainteq2} in the proof of Lemma \ref{Blambdanlemma1}. It is well known that the formula holds up to a constant, but we would like to know this constant precisely.

Let $T$ be the group of real upper triangular $n\times n$ matrices with positive diagonal entries. Then $T=AN_1$, where $A$ is the group of $n\times n$ diagonal matrices with positive diagonal entries, and $N_1$ is the group of $n\times n$ upper triangular matrices with $1$'s on the diagonal. We will put the following left-invariant Haar measure on $T$,
\begin{equation}\label{Tmeasureeq}
 \int\limits_T\phi(h)\,dh=2^n\int\limits_A\int\limits_{N_1}\phi(an)\,dn\,da,
\end{equation}
where $dn$ is the Lebesgue measure, and $da=\frac{da_1}{a_1}\ldots\frac{da_n}{a_n}$ for $a={\rm diag}(a_1,\ldots,a_n)$. Let $P_+$ be the set of positive definite $n\times n$ matrices. We endow $P_+$ with the Lebesgue measure $dY$, which also occurs in \eqref{Ghaarpropeq1}.
\begin{lemma}\label{TPpluslemma}
 The map
 \begin{equation}\label{TPpluslemmaeq1}
  \alpha:T\longrightarrow P_+,\qquad h\longmapsto h\,^th
 \end{equation}
 is an isomorphism of smooth manifolds. For a measurable function $\varphi$ on $P_+$, we have
 \begin{equation}\label{TPpluslemmaeq2}
  \int\limits_T\varphi(h\,^th)\,dh=\int\limits_{P_+}\varphi(Y)\det(Y)^{-\frac{n+1}2}\,dY.
 \end{equation}
 Here $dh$ is the measure defined by \eqref{Tmeasureeq}, and  $dY$ is the Lebesgue measure on the open subset $P_+$ of ${\rm Sym}_n(\R)$.
\end{lemma}
\begin{proof}
By Proposition 5.3 on page 272/273 of \cite{Helgason1978}, the map \eqref{TPpluslemmaeq1} is a diffeomorphism. The proof of formula \eqref{TPpluslemmaeq2} is an exercise using the tranformation formula from multivariable calculus.
\end{proof}

\begin{proposition}\label{Iwasawameasureprop}
 Let $dh$ be the Haar measure on $\Sp_{2n}(\R)$ characterized by the property \eqref{Ghaarpropeq1}. Then, for any measurable function $F$ on $\Sp_{2n}(\R)$,
 \begin{equation}\label{Sp2niwasawainteq2b}
  \int\limits_{\Sp_{2n}(\R)}F(h)\,dh=2^n\int\limits_A\int\limits_N\int\limits_KF(ank)\,dk\,dn\,da,
 \end{equation}
 where $N$ is the unipotent radical of the Borel subgroup, $dn$ is the Lebesgue measure, $A=\{{\rm diag}(a_1,\ldots,a_n,a_1^{-1},\ldots,a_n^{-1})\,:\,a_1,\ldots,a_n>0\}$, and $da=\frac{da_1}{a_1}\ldots\frac{da_n}{a_n}$.
\end{proposition}
\begin{proof}
It is well known that the right hand side defines a Haar measure $d'h$ on $\Sp_{2n}(\R)$. To prove that $d'h=dh$, it is enough to consider $K$-invariant functions $F$. For such an $F$, let $f$ be the corresponding function on $\H_n$, i.e., $f(gI)=F(g)$. Let $N_1$ and $N_2$ be as in \eqref{Blambdanlemma1eq5}, so that $N=N_1N_2$. By identifying elements of $A$ and $N_1$ with their upper left block, our notations are consistent with those used in \eqref{Tmeasureeq}. We calculate
\begin{align*}
 \int\limits_GF(g)\,d'g&=2^n\int\limits_A\int\limits_{N_1}\int\limits_{N_2}\int\limits_KF(an_1n_2k)\,dk\,dn_2\,dn_1\,da\\
  &=2^n\int\limits_A\int\limits_{N_1}\int\limits_{N_2}F(n_2an_1)\det(a)^{-(n+1)}\,dn_2\,dn_1\,da.
\end{align*}
In the last step we think of $a$ as its upper left $n\times n$ block when we write $\det(a)$. Continuing, we get
\begin{align*}
 \int\limits_GF(g)\,d'g&=2^n\int\limits_A\int\limits_{N_1}\int\limits_{{\rm Sym}_n(\R)}F(\mat{1}{X}{}{1}an_1)\det(a)^{-(n+1)}\,dX\,dn_1\,da\\
 &=2^n\int\limits_A\int\limits_{N_1}\int\limits_{{\rm Sym}_n(\R)}f(X+an_1I)\det(a)^{-(n+1)}\,dX\,dn_1\,da\\
 &=2^n\int\limits_A\int\limits_{N_1}\int\limits_{{\rm Sym}_n(\R)}f(X+i(an_1)\,^t(an_1))\det(a)^{-(n+1)}\,dX\,dn_1\,da.
\end{align*}
In the last step, again, we identify $a$ and $n_1$ with their upper left blocks. By \eqref{Tmeasureeq}, we obtain
\begin{align*}
 \int\limits_GF(g)\,d'g&=\int\limits_T\int\limits_{{\rm Sym}_n(\R)}f(X+ih\,^th)\det(h)^{-(n+1)}\,dX\,dh\\
 &=\int\limits_T\int\limits_{{\rm Sym}_n(\R)}f(X+ih\,^th)\det(h\,^th)^{-\frac{n+1}2}\,dX\,dh\\
 &=\int\limits_{P_+}\int\limits_{{\rm Sym}_n(\R)}f(X+iY)\det(Y)^{-(n+1)}\,dX\,dY.
\end{align*}
where in the last step we applied Lemma \ref{TPpluslemma}. Using \eqref{Ghaarpropeq1}, we see
$\int\limits_GF(g)\,d'g=\int\limits_GF(g)\,dg$, as asserted.
\end{proof}

\end{appendix}

\addcontentsline{toc}{section}{Bibliography}
\bibliography{pullback}{}
\end{document}